\documentclass[11pt]{article}

\usepackage{mathrsfs}
\usepackage{graphicx,  amsmath,amsfonts,titlesec,cite,color,
                                                                                   }
 \usepackage{autobreak}       
 \allowdisplaybreaks[4]   

\def\R{{\mathbb{R}}}

\def\loc{{\mathrm{loc}}}
\DeclareMathOperator*{\essinf}{ess\, inf}
\DeclareMathOperator*{\esssup}{ess\, sup}
\def\dint{\displaystyle\int}

\DeclareMathOperator{\sgn}{sgn}

\DeclareMathOperator*{\supp}{supp}

\numberwithin{equation}{section} 

\newcommand{\mathe}{\mathrm{e}}
\newcommand{\mathi}{{\mathbf i}}
\usepackage{indentfirst} 
\usepackage[amsmath,thmmarks]{ntheorem} 

\usepackage{extarrows}  
 \usepackage[sort&compress,numbers]{natbib}  
\usepackage{hyperref}
 \hypersetup{
    colorlinks=true,      
    citecolor=magenta,    
}
 \usepackage[capitalise]{cleveref}
 \usepackage[shortlabels]{enumitem}  
\def\corrauth{\footnote{Corresponding author.
 }
 \stepcounter{footnote}
 }
\def\affilnum#1{${}^{#1}$}
\def\affil#1{${}^{#1}$}
\def\comma{$^{\textrm{,}}$}

\theoremstyle{definition} 
\theoremheaderfont{\normalfont\rmfamily\bf}
\theorembodyfont{\normalfont\rm } \theoremindent0em
\newtheorem{theorem}{\indent
                  Theorem}[section]
    \newtheorem{lemma}{\indent  Lemma} [section]

    \newtheorem{definition}{\indent  Definition} [section]
        \newtheorem{remark}{\indent  Remark}  

\theoremstyle{nonumberplain} 
\theoremheaderfont{\bf\rmfamily} \theorembodyfont{\normalfont \rm }
\theoremsymbol{\hfill $\Box$} 
\newtheorem{proof}{\indent Proof}

\makeatletter
\renewcommand*{\author}[2][?]{
     \gdef\shortauthor{?}
     \gdef\@author{#2}
   \ifthenelse{\equal{#1}{?}}
     { \gdef\shortauthor{\let\comma=\empty \let\corrauth=\empty \renewcommand{\affil}[1]{} #2} }
     { \gdef\shortauthor{#1}}
}
\def\@setauthor{\begin{center}
  \sc  \@author
    \end{center}%
}

\makeatother
\title{\bf\Large
Characterization of Lipschitz Functions via Commutators of Multilinear Singular Integral Operators in\\ Variable Lebesgue Spaces  }
\author[J. WU and P. Zhang] { Jianglong Wu\affil{1} and Pu Zhang\affil{1}\comma\corrauth 
\\ 
{\affilnum{1}\footnotesize\it Department of Mathematics, Mudanjiang Normal University, Mudanjiang, 157011, China}
}
\date{} 
\begin{document}

\maketitle


\footnote{\textit{Foundation item}:\  Supported  by NNSF-China (Grant No.11571160) and MNU (No.D211220637).}   
\vspace{-4em}
\renewcommand\abstractname{}
\begin{abstract}
\setlength{\parindent}{0pt}\setlength{\parskip}{1.5explus0.5exminus0.2ex}%
\noindent
{\textbf{Abstract}:} In this paper, the main aim  is to consider the boundedness of commutators of multilinear Calder\'{o}n-Zygmund operators with Lipschitz functions  in the context of the variable exponent Lebesgue spaces. Furthermore,   the variable versions of the Lipschitz spaces are also discussed, and  with the help of a key tool  of pointwise estimate involving the  sharp maximal operator of the multilinear fractional commutator and certain associated maximal operators.


 { \bf AMS(2010) Subject Classification:}  \ 47B47; 42B20; 42B35   \ \
\end{abstract}

\section{Introduction}

Let $T$ be the classical singular integral operator. The commutator $[b, T]$ generated by $T$ and a
suitable function $b$ is defined by
\begin{align*} 
 [b,T]f      & = bT(f)-T(bf).
\end{align*}

It is well known that the commutators are intimately related to the regularity properties of the solutions of certain partial differential equations (PDE, see for example   \cite{difazio1993interior,bramanti1995commutators,rios2003lp}). 
The continuity properties of such commutators, studied in several literatures, have contributed to the development of the PDF's (such as \cite{bernardis2006weighted,cruz2003endpoint,perez1995endpoint,perez2014end} ). 

The first result for the commutator $[b,T]$ was established by   Coifman, Rochberg and Weiss in \cite{coifman1976factorization}, and the authors proved that the BMO is characterized by the boundedness of the singular integral operators' commutator $[b,T]$. In 1978,  Janson \cite{janson1978mean} generalized the results in \cite{coifman1976factorization} to functions belonging to a Lipschitz functional space and gave a characterization in terms of the boundedness of the commutators of singular integral operators with  Lipschitz functions. In 1982, Chanillo \cite{chanillo1982note} proved that BMO can be characterized by mean of the boundedness between Lebesgue spaces of the commutators
of fractional integral operators with BMO functions. In 1995, Paluszy{\'n}ski \cite{paluszynski1995characterization}  gives
some results in the spirit of \cite{chanillo1982note} for the functions belonging to Lipschitz function spaces.

The multilinear Calder\'{o}n-Zygmund theory was first studied by Coifman and Meyer in \cite{coifman1975commutators,coifman1978commutateurs}. This theory was then further investigated by many authors in the last few decades, see for example \cite{grafakos2002multilinear,lerner2009new,grafakos2014multilinear}, for the theory of multilinear Calder\'{o}n-Zygmund  operators with kernels satisfying the standard estimates.  

In 2009, Lerner et al. \cite{lerner2009new} developed a multiple-weight theory that adapts to the multilinear Calder\'{o}n-Zygmund operators.
They established the multiple-weighted norm inequalities for the multilinear Calder\'{o}n-Zygmund operators and their commutators.

The theory of function spaces with variable exponent has been intensely investigated in the past twenty years since some elementary properties were established by Kov{\'a}{\v{c}}ik  and R{\'a}kosn{\'\i}k in \cite{kovavcik1991spaces}. In 2003, Diening and R\r{u}\u{z}i\u{c}ka \cite{diening2003calderon} studied the Calder\'{o}n-Zygmund operators on variable exponent Lebesgue spaces and gave some applications to problems related to fluid dynamics. In 2006, by applying the theory of weighted norm inequalities and extrapolation, Cruz-Uribe et al. \cite{cruz2006theboundedness} showed that many classical operators in harmonic analysis are bounded on the variable exponent Lebesgue space. For more information on function spaces with variable exponent, we refer to \cite{diening2011lebesgue,cruz2013variable}.

Motivated by the work mentioned above, the main aim of this paper is to consider the boundedness of commutators of multilinear Calder\'{o}n-Zygmund operators with Lipschitz functions  in the context of the variable exponent Lebesgue spaces. In addition, we also consider the variable versions of the Lipschitz spaces, studied in \cite{ramseyer2013lipschitz,cabral2016extrapolation,pradolini2017characterization}.

Let  $\mathbb{R}^{n}$ be an $n$-dimensional Euclidean space and $(\mathbb{R}^{n})^{m}= \mathbb{R}^{n} \times \cdots \times \mathbb{R}^{n}$ be an $m$-fold product space ($m\in  \mathbb{N}$).
We denote by $\mathscr{S}(\mathbb{R}^{n})$ the space of all Schwartz functions on  $\mathbb{R}^{n}$ and by $\mathscr{S}'(\mathbb{R}^{n})$ its dual space, the set of all tempered distributions on  $\mathbb{R}^{n}$.

 \begin{definition}\label{def:CZK}  
 A locally integrable function $K(x, y_{1},\dots , y_{m})$, defined away from the diagonal $x = y_{1} = \cdots = y_{m} $ in
$(\mathbb{R}^{n})^{m+1}$ , is called an $m$-linear  Calder\'{o}n-Zygmund kernel, if there exists a constant $A > 0$ such that the following conditions are satisfied.
\begin{enumerate}[leftmargin=2em,label=(\arabic*),itemindent=1.5em]  
\item  Size estimate: for   all  $(x,y_{1},\dots , y_{m})\in (\mathbb{R}^{n})^{m+1}$ with $x\neq y_{j}$ for some $j\in\{1,2,\dots,m\}$, there has
 \begin{align} \label{equ:CZK-1}
|K(x,y_{1},\dots , y_{m})|  &\le \dfrac{A}{(|x-y_{1}|+\cdots+|x-y_{m}|)^{mn}}.
\end{align}
\item    Smoothness estimates: assume that for some $\epsilon>0$, and for each $j\in\{1,2,\dots,m\}$, there are regularity conditions
  \begin{align} \label{equ:CZK-2}    
|K(x,y_{1},\dots,y_{j}, \dots , y_{m}) -K(x',y_{1},\dots,y_{j}, \dots , y_{m}) |  &\le \dfrac{A|x-x'|^{\epsilon}}{ \Big (\sum\limits_{j=1}^{m}|x-y_{j}| \Big)^{mn+\epsilon}}
\end{align}
whenever $|x-x'| \le \dfrac{1}{2}   \max\limits_{1\le j\le m} |x- y_{j}|$,
and  for each fixed $j$ with $1\le j \le m$,
 \begin{equation} \label{equ:CZK-3}    
|K(x,y_{1},\dots,y_{j}, \dots , y_{m}) -K(x,y_{1},\dots,y_{j}', \dots , y_{m}) | \le \dfrac{A|y_{j}-y_{j}'|^{\epsilon}}{ \Big (\sum\limits_{j=1}^{m}|x-y_{j}| \Big)^{mn+\epsilon}}
\end{equation}
whenever $|y_{j}-y_{j}'| \le \dfrac{1}{2}  \max\limits_{1\le j\le m} |x- y_{j}|$.
\end{enumerate}
\end{definition}

We say $T:\mathscr{S}(\mathbb{R}^{n})\times\cdots \times\mathscr{S}(\mathbb{R}^{n}) \to \mathscr{S}'(\mathbb{R}^{n})$ is an $m$-linear singular integral operator with an $m$-linear  Calder\'{o}n-Zygmund kernel, $K(x,y_{1},\dots , y_{m})$, if
\begin{align*}
T(f_{1},\dots , f_{m})(x)  &= \dint_{(\mathbb{R}^{n})^{m}} K(x,y_{1},\dots , y_{m}) f_{1}(y_{1})\cdots f_{m}(y_{m}) dy_{1} \cdots dy_{m}
\end{align*}
whenever $x\notin \bigcap_{j=1}^{m} \supp f_{j}$ and each $f_{j} \in C_{c}^{\infty}(\mathbb{R}^{n}), j=1,\dots,m$.

If $T$ is bounded from  $L^{p_1}(\mathbb{R}^{n}) \times L^{p_2}(\mathbb{R}^{n}) \times\cdots\times L^{p_m}(\mathbb{R}^{n}) $ to $L^{p}(\mathbb{R}^{n})$ with $1<p_1,\dots,p_m<\infty$ and $\frac{1}{p}= \frac{1}{p_{_{1}}} + \frac{1}{p_{_{2}}}+\cdots+\frac{1}{p_{_{m}}}$, then we say that $T$ is an $m$-linear  Calder\'{o}n-Zygmund operator
(see \cite{grafakos2002multilinear,huang2010multilinear,lu2014multilinear} for more details).
If  $ K(x,y_{1},\dots , y_{m})$ is of form  $ K(x-y_{1},\dots , x-y_{m})$, then  $T$ is called an operator of convolution type.

Let $\vec{b} =(b_{1},b_{2},\dots,b_{m})$ be a collection of locally integrable functions, the $m$-linear commutator of $T$ with $\vec{b}$ is defined by
\begin{align*}
T_{_{\Sigma \vec{b}}}(\vec{f})(x) &= T_{_{\Sigma \vec{b}}}(f_{1},\dots , f_{m})(x)  = \sum_{j=1}^{m} T_{_{b_{j}}}(\vec{f})(x),
\end{align*}
where each term is the commutator of $b_{j}$ and $T$ in the $j$-th entry of $T$ , that is,
\begin{align*}
T_{_{b_{j}}}(\vec{f})(x) &= [b_{j},T](\vec{f}) (x) =b_{j}(x)T(f_{1},\dots,f_{j} ,\dots, f_{m})(x)   -T(f_{1},\dots,b_{j}f_{j} ,\dots, f_{m})(x)
\end{align*}
for every $j=1,2,\dots,m$. This definition coincides with the linear commutator $[b,T ]$ when $m = 1$.
And the iterated commutator $T_{_{\Pi \vec{b}}}(\vec{f})$ is defined via \cite{perez2014end}
\begin{align*}
T_{_{\Pi \vec{b}}}(\vec{f})(x) &= [b_{1},[b_{2},\dots[b_{m-1},[b_{m},T]_{m}]_{m-1}\dots]_{2}]_{1}(\vec{f}) (x).
\end{align*}
To clarify the notation, the commutators can be wrote formally as  
\begin{align*}
T_{_{\Sigma \vec{b}}}(\vec{f})(x)  &= \sum_{j=1}^{m} \int_{(\mathbb{R}^{n})^{m}} (b_{j}(x)-b_{j}(y_{j})) K(x,y_{1},\dots , y_{m}) f_{1}(y_{1})\cdots f_{m}(y_{m}) dy_{1} \cdots dy_{m}   \\
&= \sum_{j=1}^{m}  \int_{(\mathbb{R}^{n})^{m}}  (b_{j}(x)-b_{j}(y_{j})) K(x,\vec{y}) \prod_{i=1}^{m} f_{i}(y_{i})  d\vec{y};\\
T_{_{\Pi \vec{b}}}(\vec{f})(x) &= \int_{(\mathbb{R}^{n})^{m}} \left(\prod_{j=1}^{m}(b_{j}(x)-b_{j}(y_{j})) \right) K(x,y_{1},\dots , y_{m}) f_{1}(y_{1})\cdots f_{m}(y_{m}) dy_{1} \cdots dy_{m}  \\
&= \int_{(\mathbb{R}^{n})^{m}} \left( \prod_{j=1}^{m}(b_{j}(x)-b_{j}(y_{j})) \right) K(x,\vec{y}) \prod_{i=1}^{m} f_{i}(y_{i})  d\vec{y}.
\end{align*}

When $m=1$, $T_{_{\Sigma \vec{b}}}(\vec{f}) = T_{_{\Pi \vec{b}}}(\vec{f}) = [b,T]f = bT(f) - T(bf)$, which is the well  known  classical commutator studied in \cite{coifman1976factorization}. These multilinear commutators are early appeared in  \cite{xu2006generalized}.

Throughout this paper, the letter $C$  always stands for a constant  independent of the main parameters involved and whose value may differ from line to line.
A cube $Q \subset\mathbb{R}^{n} $ always means a cube whose sides are parallel to the coordinate axes and denote its side length by $l(Q)$.
For some $t>0$, the  notation  $tQ$ stands   for  the cube with the same center as  $Q$ and with side length  $l(tQ)=t l(Q)$.
 Denote by $|S|$ the Lebesgue measure  and  by $\chi_{_{\scriptstyle S}}$ the characteristic function  for a measurable set  $S\subset\mathbb{R}^{n}$. $B(x,r)$ means the ball cenetered at $x$ and of radius $r$, and $B_{0}=B(0,1)$.
For any index $1< q(x)< \infty$, we denote by $q'(x)$ its conjugate index,
namely, $q'(x)=\frac{q(x)}{q(x)-1}$. And we will occasionally use the notational $\vec{f}=(f_{1},\dots , f_{m})$, $T(\vec{f})=T(f_{1},\dots , f_{m})$, $d\vec{y}=dy_{1}\cdots  dy_{m}$ and $(x,\vec{y})=(x,y_{1},\dots , y_{m})$ for convenience.
For a set $E$ and a positive integer $m$, we will use the notation $(E)^{m}=\underbrace{E\times \cdots \times E}_{m}$ sometimes.


\section{Preliminaries}

Over last three decades, the study of variable exponent function spaces have attracted many authors' attention(see \cite{cruz2006theboundedness,diening2011lebesgue,cruz2013variable,diening2003calderon}   et al.). In fact, many classical operators are discussed in variable exponent function spaces(see \cite{cruz2006theboundedness,diening2011lebesgue,cruz2013variable}).  

In this section, we give the definition of Lebesgue spaces with variable exponent, and state basic properties
and useful lemmas.

\subsection{Function spaces with variable exponent}

Let $\Omega$ be a measurable set in $\R^{n}$ with $|\Omega|>0$. We first define variable exponent Lebesgue spaces.

\begin{definition} \label{def.2.1}
\ \ Let ~$ q(\cdot): \Omega\to[1,\infty)$ be a measurable function.

 \begin{list}{}{}  
\item [(i)]  \ The Lebesgue spaces with variable exponent $L^{q(\cdot)}(\Omega)$ is defined by
  $$ L^{q(\cdot)}(\Omega)=\{f~ \mbox{is measurable function}:  F_{q}(f/\eta)<\infty ~\mbox{for some constant}~ \eta>0\}, $$
  where $F_{q}(f):=\int_{\Omega} |f(x)|^{q(x)} \mathrm{d}x$. The Lebesgue space $L^{q(\cdot)}(\Omega)$ is a Banach function space with respect to the norm
  \begin{equation*}
   \|f\|_{L^{q(\cdot)}(\Omega)}=\inf \Big\{ \eta>0:  F_{q}(f/\eta)=\int_{\Omega} \Big( \frac{|f(x)|}{\eta} \Big)^{q(x)} \mathrm{d}x \le 1 \Big\}.
\end{equation*}
 \item[(ii)] \ The space $L_{\loc}^{q(\cdot)}(\Omega)$ is defined by
  $$ L_{\loc}^{q(\cdot)}(\Omega)=\{f ~\mbox{is measurable}: f\in L^{q(\cdot)}(\Omega_{0}) ~\mbox{for all compact  subsets}~ \Omega_{0}\subset \Omega\}.  $$

 \item[(iii)] \ The weighted Lebesgue space $L_{\omega}^{q(\cdot)}(\Omega)$ is defined by as the set of all measurable functions for which $$\|f\|_{L^{q(\cdot)}_{\omega}(\Omega)}=\|\omega f\|_{L^{q(\cdot)}(\Omega)}<\infty.$$
\end{list}
\end{definition}

Next we define some classes of variable  exponent functions. Given a function $f\in L_{\loc}^{1}(\R^{n})$, the Hardy-Littlewood maximal operator $M$ is defined by
$$Mf(x)= \sup_{Q\ni x} \dfrac{1}{|Q| } \int_{Q}  |f(y)| \mathrm{d}y. $$

\begin{definition} \label{def.variable-exponent}\ \
 \ Given a measurable function $q(\cdot)$ defined on $\R^{n}$. For $E\subset \mathbb{R}^{n}$, we write
$$
q_{-}(E):=\essinf_{x\in E} q(x),\ \
q_{+}(E):= \esssup_{x\in E} q(x),$$
and write $q_{-}(\mathbb{R}^{n}) = q_{-}$ and $q_{+}(\mathbb{R}^{n}) = q_{+}$ simply.

\begin{list}{}{}
\item[(i)]\   $q'_{-}=\essinf\limits_{x\in \mathbb{R}^{n}} q'(x)=\frac{q_{+}}{q_{+}-1},\ \ q'_{+}= \esssup\limits_{x\in \R^{n}} q'(x)=\frac{q_{-}}{q_{-}-1}.$

\item[(ii)]\ Denote by $\mathscr{P}_{0}(\mathbb{R}^{n})$
the set of all measurable functions $ q(\cdot): \mathbb{R}^{n}\to(0,\infty)$ such that
$$0< q_{-}\le q(x) \le q_{+}<\infty,\ \ x\in \mathbb{R}^{n}.$$

\item[(iii)]\ Denote by $\mathscr{P}_{1}(\mathbb{R}^{n})$
the set of all measurable functions $ q(\cdotp): \mathbb{R}^{n}\to[1,\infty)$ such that
$$1\le q_{-}\le q(x) \le q_{+}<\infty,\ \ x\in \mathbb{R}^{n}.$$

\item[(iv)]\ Denote by $\mathscr{P}(\R^{n})$ the set of all measurable functions $ q(\cdot): \R^{n}\to(1,\infty)$ such that
$$1< q_{-}\le q(x) \le q_{+}<\infty,\ \ x\in \R^{n}.$$

\item[(v)]\  The set $\mathscr{B}(\R^{n})$ consists of all  measurable functions  $q(\cdot)\in\mathscr{P}(\R^{n})$ satisfying that the Hardy-Littlewood maximal operator $M$ is bounded on $L^{q(\cdot)}(\R^{n})$.

\end{list}
\end{definition}

\begin{definition} [$\log$-H\"{o}lder continuity] \label{def.log-holder} \ \
 \ Let $q(\cdot)$ be a real-valued function on  $\mathbb{R}^{n}$.

\begin{list}{}{}
\item[(i)]\    Denote by  $\mathscr{C}^{\log}_{loc}(\mathbb{R}^{n})$
the set of all local  $\log$-H\"{o}lder continuous functions $q(\cdotp)$ which satisfies
 \begin{equation*}    
 |q(x)-q(y)| \le \frac{-C}{\ln(|x-y|)},  \ \ \  |x-y|\le 1/2,\ x,y \in \R^{n},
\end{equation*}
where $C$ denotes a universal positive constant that may differ from line to line, and  $C$ does not depend on $x, y$.
%

\item[(ii)]\  The set $\mathscr{C}^{\log}_{\infty}(\R^{n})$ consists of all  $\log$-H\"{o}lder continuous functions $ q(\cdot)$ at infinity   satisfies
\begin{equation*}
 |q(x)-q_{\infty}| \le \frac{C_{\infty}}{\ln(\mathe+|x|)},  \ \ \ x \in \R^{n},
\end{equation*}
 where $q_{\infty}=\lim\limits_{|x|\to \infty}q(x)$.

\item[(iii)]\  Denote by  $\mathscr{C}^{\log}(\R^{n}):=\mathscr{C}^{\log}_{loc}(\R^{n})\cap \mathscr{C}^{\log}_{\infty}(\R^{n})$ the set of all global $\log$-H\"{o}lder continuous functions $ q(\cdot)$.

\end{list}

\end{definition}

\begin{remark}\quad The $\mathscr{C}^{\log}_{\infty}(\R^{n})$ condition is equivalent to the uniform continuity condition
\begin{equation*}  
  |q(x)-q(y)| \le \frac{C}{\ln(\mathe+|x|)},  \ \ \  |y|\ge|x|,\ x,y \in \R^{n}.
\end{equation*}
The $\mathscr{C}^{\log}_{\infty}(\R^{n})$ condition was originally defined in this form in \cite{cruz2003maximal}.
\end{remark}

\subsection{Auxiliary propositions  and lemmas }

 In this part we   state some auxiliary propositions  and lemmas which will be needed for proving  our main theorems. And we only describe partial results  we need.

 \begin{lemma}\quad \label{lem.variable-proposition-B}
Let $ p(\cdot)\in \mathscr{P}(\R^{n})$.
 \begin{list}{}{}
 \item[(1)]\ If $ p(\cdot)\in \mathscr{C}^{\log}(\R^{n})$, 
 then we have $ p(\cdot)\in \mathscr{B}(\R^{n})$.
\item[(2)] \cite[see Lemma 2.3 in ][]{cruz2014variable}
The  following conditions are equivalent:
\begin{enumerate}[label=(\roman*),itemindent=1em]  
 \item \   $ p(\cdot)\in \mathscr{B}(\R^{n})$,
 \item \   $p'(\cdot)\in \mathscr{B}(\R^{n})$.
  \item \   $ p(\cdot)/p_{0}\in \mathscr{B}(\R^{n})$ for some $1<p_{0}<p_{-}$,
 \item \   $ (p(\cdot)/p_{0})'\in \mathscr{B}(\R^{n})$ for some $1<p_{0}<p_{-}$.
\end{enumerate}
\end{list}
\end{lemma}

 The  first part in   \cref{lem.variable-proposition-B} is independently due to Cruz-Uribe et al. \cite{cruz2003maximal} and to Nekvinda \cite{nekvinda2004hardy}  respectively. The second of  \cref{lem.variable-proposition-B}  belongs to  Diening \cite{diening2005maximalf}~(see Theorem 8.1 or Theorem 1.2 in  \cite{cruz2006theboundedness}).

\begin{remark} \label{rem.2}
\begin{enumerate}[leftmargin=2em,label=(\alph*),itemindent=1.5em]  
\item   Since
$$|q'(x)-q'(y)|\le \frac{|q(x)-q(y)|}{(q_{-}-1)^{2}},$$
 it follows at once that if $q(\cdot)\in \mathscr{C}^{\log}(\mathbb{R}^{n})$, then so does $q'(\cdot)$, that is, if the condition hold, then $M$ is bounded on
  $L^{q(\cdot)} (\mathbb{R}^{n})$ and $L^{q'(\cdot)} (\mathbb{R}^{n})$.  Furthermore, Diening has proved general results on Musielak-Orlicz spaces.
\item   When $p(\cdot)\in \mathscr{P}(\mathbb{R}^{n})$,  the assumption that $p(\cdot)\in \mathscr{C}^{\log}(\mathbb{R}^{n})$ is equivalent to assuming  $1/p(\cdot)\in \mathscr{C}^{\log}(\mathbb{R}^{n})$, since
$$\Big|\frac{p(x)-p(y)}{(p_{+})^{2}}\Big| \le \Big|\frac{1}{p(x)}-\frac{1}{p(y)} \Big|=\Big|\frac{p(x)-p(y)}{p(x)p(y)}\Big|\le \Big|\frac{p(x)-p(y)}{(p_{-})^{2}}\Big|.$$
\end{enumerate}
\end{remark}

As the classical Lebesgue norm, the (quasi-)norm of variable exponent Lebesgue space is also homogeneous in the
exponent. Precisely, we have the following result \cite[see Lemma 2.3 in ][]{cruz2014variable}.

\begin{lemma}  \label{lem.variable-homogeneous}
Given   $ p(\cdot)\in    \mathscr{P}_{0}(\R^{n})$, then for all   $s>0$, we have
\begin{align*}
 \||f|^{s}\|_{p(\cdot)}  &\le C  \|f\|_{sp(\cdot)}^{s}.
\end{align*}
\end{lemma}


The next lemma is known as  the generalized H\"{o}lder's inequality on Lebesgue spaces with
variable exponent, and  the proof can  also be found in \cite{kovavcik1991spaces} or \cite[P.27-30 in][]{cruz2013variable}.

\begin{lemma} [generalized H\"{o}lder's inequality]\label{lem.holder} \
\begin{list}{}{} 
\item[(1)]\cite[see P.81-82, Lemma 3.2.20 in ][]{diening2011lebesgue}\ \ Let $p(\cdotp),q(\cdotp),r(\cdotp)\in  \mathscr{P}_{0}(\mathbb{R}^{n})$ satisfy the condition
\[
\dfrac{1}{r(x)} = \dfrac{1}{p(x)} + \dfrac{1}{q(x)} \qquad \mbox{for a.e.} \ x\in \mathbb{R}^{n}.
\]
\begin{enumerate}[label=(\roman*),itemindent=1em]  
\item Then, for all $f \in L^{{p(\cdotp)}}(\mathbb{R}^{n})$ and $g\in L^{{q(\cdotp)}}(\mathbb{R}^{n})$, one has
\begin{align} \label{equ:holder-2}
\|fg\|_{r(\cdotp)} &\le C\|f\|_{p(\cdotp)}  \|g\|_{q(\cdotp)}.
\end{align}

\item   When $r=1$, then $p'(\cdotp) = q(\cdotp)$, hence, for all $f \in L^{{p(\cdotp)}}(\mathbb{R}^{n})$ and $g\in L^{{p'(\cdotp)}}(\mathbb{R}^{n})$, one has
\begin{align}\label{equ:holder-1}
\int_{\mathbb{R}^{n}}|fg|  &\le C\|f\|_{p(\cdotp)}  \|g\|_{p'(\cdotp)}.
\end{align}
\end{enumerate}
\item[(2)]The generalized H\"{o}lder's inequality in Orlicz space\cite[for  details and  the more general cases, ][]{perez2002sharp,perez1995endpoint,lerner2009new}.\ 
\begin{enumerate}[label=(\roman*),itemindent=1em]  
\item Let $r_1,\dots,r_m\ge 1$ with $\frac{1}{r}=\frac{1}{r_1}+\cdots+\frac{1}{r_m}$ and $Q$ be a cube in $\mathbb{R}^{n}$. Then
\begin{align*}
\frac{1}{|Q|}\int_{Q}|f_{1}(x)\cdots f_{m}(x)g(x)| dx &\le C\|f_{1}\|_{\exp L^{r_{1}},Q} \cdots \|f_{m}\|_{\exp L^{r_{m}},Q} \|g\|_{L(\log L)^{1/r},Q}.
\end{align*}

\item   Let $t\ge 1$, then
\begin{align}\label{equ:holder-4}
\frac{1}{|Q|}\int_{Q}|f(x)g(x)| dx &\le C\|f\|_{\exp L^{t},Q}   \|g\|_{L(\log L)^{1/t},Q}.
\end{align}
\end{enumerate}
\item[(3)]\cite[see Lemma 9.2 in ][]{lu2014multilinear}\ Let $q(\cdotp),q_{1}(\cdot),\dots, q_{m}(\cdot)\in  \mathscr{P}(\mathbb{R}^{n})$ satisfy the condition
\[
\dfrac{1}{q(x)} = \dfrac{1}{q_{1}(x)}+\cdots + \dfrac{1}{q_{m}(x)} \qquad \mbox{for a.e.} \ x\in \mathbb{R}^{n}.
\]
Then, for any $f_{j} \in L^{{q_{j}(\cdotp)}}(\mathbb{R}^{n})$ , $j=1,\dots,m$, one has
\begin{align*} 
\|f_{1}\cdots f_{m} \|_{q(\cdotp)} &\le C\|f_{1}\|_{q_{1}(\cdot)} \cdots \|f_{m}\|_{q_{m}(\cdot)}.
\end{align*}
\end{list}
\end{lemma}

\begin{lemma}\cite[see Corollary 2.1  in ][]{huang2010multilinear} \label{lem.CZO-bound}
Let $T$ be a 2-linear Calder\'{o}n-Zygmund operator. Suppose that $ p_{_{1}}(\cdot),~p_{_{2}}(\cdot) \in \mathscr{B}(\R^{n})$. If  $ p(\cdot)\in    \mathscr{P}_{0}(\R^{n})$ such that there exists $p_{*}\in (0,p_{-})$ with $(p(\cdot)/p_{*})' \in \mathscr{B}(\R^{n})$ and $\frac{1}{p(x)} = \frac{1}{p_{1}(x)} + \frac{1}{p_{2}(x)} $, then there exists a constant $C$ independent of functions $f_{i} \in L^{{p_{i}(\cdot)}}(\mathbb{R}^{n})$ for $i = 1, 2$ such that
\begin{align*} 
\|T(f_{1},f_{2}) \|_{p(\cdot)} &\le C\|f_{1}\|_{p_{1}(\cdot)} \|f_{2}\|_{p_{2}(\cdot)}.
\end{align*}
\end{lemma}

The following results are also needed. 

\begin{lemma}[Norms of characteristic functions]\label{lem.character-norm}
\ \ 
\begin{enumerate}[ label=(\arabic*),itemindent=1em] 
\item \ Let $q(\cdot)\in  \mathscr{C}^{\log}(\R^{n}) \cap\mathscr{P}(\R^{n})$ and $q(x)\le q_{\infty}$ for a.e. $x\in \mathbb{R}^{n}$. Then there exists a positive constant $C$ such that the inequality
\begin{align*}
\|\chi_{Q}\|_{q(\cdot)}    & \le C |Q|^{1/q(x)}  
\end{align*}
 holds for every cube $Q\subset \mathbb{R}^{n}$ and a.e. $x\in Q$ (see Lemma 4.4 in \cite[][]{pradolini2017characterization}, or P.126, Corollary 4.5.9 in   \cite[][]{diening2011lebesgue} ).
    \label{enumerate:charac-norm-1}
\item\ Let $q(\cdotp)\in  \mathscr{P}(\mathbb{R}^{n})$. $\frac{1}{q_{_{Q}}} = \frac{1}{|Q|} \dint_{Q} \frac{1}{q(y)} dy$ is the harmonic mean of $q_{_{Q}}$. Then the following conditions are equivalent (see  Theorem 4.5.7 in \cite[][]{diening2011lebesgue} or Proposition 4.66 in \cite{cruz2013variable}).
\begin{enumerate}[label=(\roman*),itemindent=1em]  
\item $\|\chi_{Q}\|_{q(\cdotp)} \|\chi_{Q}\|_{q'(\cdotp)}\approx |Q|$ uniformly for all cubes $Q\subset\mathbb{R}^{n}$.
\item  $\|\chi_{Q}\|_{q(\cdotp)} \approx |Q|^{^{\frac{1}{q_{Q}}}}$ and $\|\chi_{Q}\|_{q'(\cdotp)}\approx |Q|^{^{\frac{1}{q'_{Q}}}} $ uniformly for all cubes $Q\subset\mathbb{R}^{n}$.
\end{enumerate}
    \label{enumerate:charac-norm-2}
\item
    Let $q_{+}<\infty$ . Then the following conditions are equivalent (see P.101, Lemma 4.1.6 in  \cite[][]{diening2011lebesgue}, or   Lemma 4.2 in \cite{pradolini2017characterization}, or Corollary 3.24 in  \cite{cruz2013variable}, or \cite{cruz2012weighted,harjulehto2004variable}).
\begin{enumerate} [itemindent=1em]
\item   The function $q(\cdot)\in \mathscr{C}^{\log}_{0}(\mathbb{R}^{n})$ .
\item  For every cube $Q\subset\mathbb{R}^{n}$, there exists a positive constant $C$ such that
$$ |Q|^{q_{-}(Q)  - q_{+}(Q) } \le C.$$
\item  For all cube $Q\subset\mathbb{R}^{n}$ and all $x\in Q$, there exists a positive constant $C$ such that
$$ |Q|^{q_{-}(Q)  - q(x) } \le C.$$
\item  For all cube $Q\subset\mathbb{R}^{n}$ and all $x\in Q$, there exists a positive constant $C$ such that
$$ |Q|^{ q(x)- q_{+}(Q) } \le C.$$
\end{enumerate}
    \label{enumerate:charac-norm-3}
\item  
Let $q(\cdot)\in \mathscr{C}^{\log}_{\infty}(\mathbb{R}^{n})$ . Then $\|\chi_{Q}\|_{q(\cdotp)} \approx |Q|^{1/q_{\infty}}$ for every cube $Q\subset\mathbb{R}^{n}$ with diameter $r_{Q}\ge 1/4$ (a particular case of Lemma 3.6 in \cite{cruz2012weighted}, see also \cite{pradolini2017characterization}). 
    \label{enumerate:charac-norm-4}
\item    Let $q(\cdot)\in \mathscr{C}^{\log}(\R^{n}) \cap\mathscr{P}(\R^{n})$ . Then $\|\chi_{Q}\|_{q(\cdotp)} \approx |Q|^{\frac{1}{q_{Q}}}$ for every cube (or ball) $Q\subset\mathbb{R}^{n}$.  More concretely,
\begin{equation*}
\|\chi_{Q}\|_{q(\cdotp)} \approx
\left\{ \begin{aligned}
        |Q|^{\frac{1}{q(x)}} & \ \ \hbox{if } |Q|\le  2^{n} ~\hbox{and}~ x\in Q \\
        |Q|^{\frac{1}{q_{\infty}}} &  \ \ \hbox{if }|Q|\ge 1
                          \end{aligned} \right.
                          \end{equation*}
for every cube (or ball)   $Q\subset\mathbb{R}^{n}$  \cite[see   Corollary 4.5.9 and Lemma 7.3.19 in ][]{diening2011lebesgue}.
    \label{enumerate:charac-norm-5}
\item    Given a cube  $Q=Q(x_{0},r)$, with center in $x_{0}$ and diameter $r$.
\begin{enumerate} [itemindent=1em]
\item  If $r<1$, there exist two positive constants $a_{1}$ and $a_{2}$ such that
\begin{align*}
a_{1} |Q|^{1/q_{-}(Q)  } \le \|\chi_{Q}\|_{q(\cdot)}    & \le a_{2} |Q|^{1/q_{+}(Q) };  
\end{align*}
\item  If $r>1$, there exist two positive constants $c_{1}$ and $c_{2}$ such that
\begin{align*}
c_{1} |Q|^{1/q_{+}(Q) } \le \|\chi_{Q}\|_{q(\cdot)}    & \le c_{2} |Q|^{1/q_{-}(Q)  }.  
\end{align*}
\end{enumerate}
Therefore, $\|\chi_{Q}\|_{q(\cdot)} \le \max\{|Q|^{1/q_{+}(Q) }, |Q|^{1/q_{-}(Q)  }\}$ (see  P.25-26, Corollary 2.23 in  \cite[][]{cruz2013variable}, or \cite{fan2001spaces,pradolini2017characterization, diening2011lebesgue}).  
    \label{enumerate:charac-norm-6}
\item \   If  $p(\cdot)\in  \mathscr{C}^{\log}(\R^{n}) \cap\mathscr{P}(\R^{n})$  and $\beta=n/\alpha$ with $p_{+}< \frac{n\beta}{(n-\beta)^{+}}$, then there exists a number $a>1$ such that
\begin{equation}  \label{equ:character-norm-est}
  \|\chi_{Q(x,ar)}\|_{p'(\cdot)} \le \frac{a^{n-n/\beta+1}}{2} \|\chi_{Q(x,r)}\|_{p'(\cdot)}
\end{equation}
for every $r>0$ and $x\in \mathbb{R}^{n}$, where $Q(x,r)$ denotes a cube centered at $x$ and with diameter $r$ (It is easy to check that the result above can be obtained for every $\beta>1$, for more information, see Lemma 2.17 in \cite{ramseyer2013lipschitz} or \cite{pradolini2017characterization}).
    \label{enumerate:charac-norm-7}
\end{enumerate}      
\end{lemma}

%

 Let  $k_{0}\in \mathbb{N}$ satisfy that
  \begin{align}\label{equ.double-condition-2}
 a^{k_{0}-1}  &<2<a^{k_{0}},
\end{align}
where  $a$ is given in  condition \eqref{equ:character-norm-est}, and let $\epsilon_{0}=1/k_{0}$  (which will be used in the following sections).

Note that, if  $p(\cdot)\in  \mathscr{C}^{\log}(\R^{n}) \cap\mathscr{P}(\R^{n})$, the estimates (\ref{equ:character-norm-est}) and  (\ref{equ.double-condition-2}) imply the doubling condition  for the functional $a(Q):= \|\chi_{Q}\|_{p(\cdot)}$, that is
\begin{align*}
\|\chi_{2Q}\|_{p(\cdot)}    & \le C \|\chi_{Q}\|_{p(\cdot)}   
\end{align*}
 for every cube $Q\subset \mathbb{R}^{n}$.

%


Set $0<\gamma<n$ and $p(\cdot),q(\cdot)\in  \mathscr{P}(\R^{n})$ such that $1/q(x)=1/p(x) - \gamma/n$ with $p_{+}<n/\gamma$. Then a weight
$\omega\in    \mathcal{A}_{p(\cdot),q(\cdot) }^{\gamma}(\mathbb{R}^{n})$ if there exists a positive constant $C$ such that for every cube $Q$, the inequality
\begin{equation}  \label{equ:weight-apq}
  \|\omega\chi_{Q}\|_{q(\cdot)}  \|\omega^{-1}\chi_{Q}\|_{p'(\cdot)}   \le C|Q|^{1-\gamma/n}
\end{equation}
holds.

When $ \gamma=0$, the inequality above is the $ \mathcal{A}_{p(\cdot) }(\mathbb{R}^{n})$ class given by Cruz-Uribe, Diening and H\"{a}st\"{o} in \cite{cruz2011maximal}, that characterizes the boundedness of the Hardy-Littlewood maximal operator on $L_{\omega}^{p(\cdot)}(\mathbb{R}^{n})$,
 that is, the measurable functions $f$ such that $f\omega \in L^{p(\cdot)}(\mathbb{R}^{n})$.

 The following result was proved in \cite{bernardis2014generalized} and gives a relation between the  $ \mathcal{A}_{p(\cdot) }(\mathbb{R}^{n})$  and the $ \mathcal{A}_{p(\cdot),q(\cdot) }^{\gamma}(\mathbb{R}^{n})$  classes (see also Lemma 4.14 in  \cite{pradolini2017characterization}).

\begin{lemma} \cite[see Lemma 4.1 in ][]{bernardis2014generalized}\label{lem.weight-relation}
Let $0<\gamma<n$ and $\omega$ be a weight. Set $p(\cdot), q(\cdot), s(\cdot)\in   \mathscr{P}(\R^{n})$ such that $1/q(x)=1/p(x) - \gamma/n$ and $s(x)=(1-\gamma/n)q(x)$ with $p_{+}<n/\gamma$. Then $\omega\in    \mathcal{A}_{p(\cdot),q(\cdot) }^{\gamma}(\mathbb{R}^{n})$ if and only if $\omega^{\frac{n}{n-\gamma}}\in    \mathcal{A}_{s(\cdot)}(\mathbb{R}^{n})$ .
\end{lemma}

Note that if $q(\cdot)\in  \mathscr{C}^{\log}(\R^{n}) \cap\mathscr{P}(\R^{n})$, then $s(\cdot)\in  \mathscr{C}^{\log}(\R^{n})\cap\mathscr{P}(\R^{n})$. Since $M$ is continuous on $L^{s(\cdot)}(\mathbb{R}^{n})$, thus Pradolini and Ramos obtain the following lemma\cite{pradolini2017characterization}.

\begin{lemma} \cite[Lemma 4.15 in ][]{pradolini2017characterization}\label{lem.weight-apq-1}
Let $0\le\gamma<n$, $p(\cdot)\in  \mathscr{C}^{\log}(\R^{n}) \cap\mathscr{P}(\R^{n})$ and $1/q(x)=1/p(x) - \gamma/n$. Then $1\in    \mathcal{A}_{p(\cdot),q(\cdot) }^{\gamma}(\mathbb{R}^{n})$.
\end{lemma}

The following definition of the multilinear fractional integral operator was considered by several authors (see, for example,
\cite{kenig1999multilinear,moen2009weighted,pradolini2010weighted,tan2017some}).
\begin{definition} [multilinear fractional integral operator] \label{def.mul-frac-inte} \
Let $0 \le \alpha < mn$ and $\vec{f}=(f_{1},f_{2},\dots,f_{m})$. The multilinear fractional integral is defined by
\begin{align*} 
\mathcal{I}_{\alpha}(\vec{f})(x)  & = \dint_{(\mathbb{R}^{n})^{m}} \dfrac{f_{1}(y_{1}) \cdots f_{m}(y_{m})}{(|x-y_{1}|+\cdots+|x-y_{m}|)^{mn-\alpha}}  d \vec{y},
\end{align*}
where  the integral is convergent if   $\vec{f}\in \mathscr{S}(\mathbb{R}^{n})\times\cdots \times\mathscr{S}(\mathbb{R}^{n})$.
\end{definition}

The following lemma for multilinear fractional integral operators in variable Lebesgue spaces is needed,  and its proof can  be found in \cite{tan2017some}.
In addition, the weighted inequalities for multilinear fractional integral operators  has been established by Moen in classical function spaces\cite{moen2009weighted}.

\begin{lemma} \label{lem. multilinear-fractional}
Suppose that $0 < \alpha < mn$, $p_{_{1}}(\cdot),p_{_{2}}(\cdot),\dots,p_{_{m}}(\cdot)\in \mathscr{C}^{\log}(\R^{n}) \cap\mathscr{P}(\R^{n})$ satisfy $\frac{1}{p(x)}= \frac{1}{p_{_{1}}(x)} + \frac{1}{p_{_{2}}(x)}+\cdots+\frac{1}{p_{_{m}}(x)}$, and $({p_{_j}})_{+}<\frac{mn}{\alpha}~(j=1,2,\dots,m)$. Define the variable exponent $q(\cdot)$ by
$$\frac{1}{q(x)}=\frac{1}{p(x)}-\frac{\alpha}{n}.$$
Then there exists a positive constant $C$ such that
$$ \| \mathcal{I}_{\alpha}(\vec{f}) \|_{q(\cdot)} \le C \prod_{i=1}^{m} \|f_{i} \|_{p_{i}(\cdot)}.$$
\end{lemma}

\subsection{A pointwise estimate}

 The following notations can be founded in refs. \cite{pradolini2017characterization,cabral2016extrapolation}.

\begin{definition} \label{def.pointwise-sharp}\ \
Let $f$ be a locally integrable function defined on  $\mathbb{R}^{n}$.
\begin{list}{}{} 
\item[(1)]  Set $0\le \delta<1$. The $\delta$-sharp maximal operator is defined by
\begin{align}\label{equ:del-sharp}
f_{\delta}^{\sharp} (x)     &= \sup_{Q\ni x} \dfrac{1}{|Q|^{1+\delta/n}} \int_{Q} |f(y)-f_{Q}| dy,   
\end{align}
where the supremum is taken over all cube $Q \subset\mathbb{R}^{n}$ containing $x$, and $f_{Q}=|Q|^{-1} \int_{Q} f(z) dz$ denotes the average of $f$ over the cube $Q\subset \mathbb{R}^{n}$.
\item[(2)] Let $0\le \delta(\cdot)<1$, $p(\cdot)\in \mathscr{P}(\mathbb{R}^{n})$ such that  $\delta(\cdot)/n= 1/\beta-1/p(\cdot)$. The $\delta(\cdot)$-sharp maximal operator is defined by
\begin{equation} \label{equ:delta-sharp}
f_{\delta(\cdot)}^{\sharp} (x)     = \sup_{Q\ni x} \dfrac{1}{|Q|^{1/\beta-1} \|\chi_{Q}\|_{p'(\cdotp)} } \left(\frac{1}{|Q|}\int_{Q}|f(y)-f_{Q}| dy \right).
\end{equation}
\begin{enumerate}[label=(\roman*),itemindent=1em]  
\item For any $\gamma>0$,  the following generalization of the operator above is
\begin{equation*} 
f_{\delta(\cdot),\gamma}^{\sharp} (x)     = \sup_{Q\ni x} \dfrac{1}{|Q|^{1/\beta-1} \|\chi_{Q}\|_{p'(\cdotp)} } \left(\frac{1}{|Q|}\int_{Q}\Big| |f(y)|^{\gamma}-(|f|^{\gamma})_{Q} \Big| dy \right)^{1/\gamma}.
\end{equation*}
\end{enumerate}
\item[(3)]  Let $0\le\alpha<n$ and $\epsilon>0$, define the following operators via
\begin{enumerate}  [itemindent=1em]
\item[(ii)]  $M_{\epsilon}f (x)= \big[ M(|f|^{\epsilon})(x) \big]^{1/\epsilon}=\Big(\sup\limits_{Q\ni x}  \dfrac{1}{|Q|}  \dint_{Q} |f(y)|^{\epsilon} d y \Big)^{1/\epsilon}$.
\item[(iii)] $M_{\alpha,L(\log L)}f(x) = \sup\limits_{Q\ni x}  |Q|^{\alpha/n} \|f\|_{L(\log L),Q} $,
where $\|\cdot\|_{L(\log L),Q}$ is the Luxemburg type average defined by
$$ \|f\|_{L(\log L),Q} = \inf \Big\{\lambda>0:  \frac{1}{|Q|} \int_{Q}  \frac{|f(x)|}{\lambda} \log(e+|f|/\lambda)  \mathrm{d}x \le 1 \Big\}.$$
\end{enumerate}
\end{list}
\end{definition}

The following multilinear maximal functions that adapts to the multilinear Calder\'{o}n-Zygmund theory are introduced by Lerner et al. in \cite{lerner2009new}.

\begin{definition} [multilinear maximal functions, see  \cite{lu2014multilinear}]
\label{def.mul-max} \
For all locally integrable functions $\vec{f}=(f_{1},f_{2},\dots,f_{m})$ and $x\in \mathbb{R}^n$,
\begin{list}{}{} 
\item[(1)] the multilinear maximal functions $\mathcal{M}$ and $\mathcal{M}_{r}$ are defined by
\begin{enumerate}[label=(\roman*),itemindent=1em]  
\item
$\mathcal{M}(\vec{f})(x)   = \sup\limits_{Q\ni x} \prod\limits_{j=1}^{m} \dfrac{1}{|Q|}  \dint_{Q} |f_{j}(y_{j})| d y_{j}$,

\item
$\mathcal{M}_{r}(\vec{f})(x)   = \sup\limits_{Q\ni x} \prod\limits_{j=1}^{m} \Big(\dfrac{1}{|Q|}  \dint_{Q} |f_{j}(y_{j})|^{r} d y_{j} \Big)^{1/r}$, for $r>1$,
\end{enumerate}
\item[(2)] the maximal functions related to Young function $\Phi(t)=t(1+\log^{+}t)$ are defined by
\begin{enumerate}[label=(\roman*),itemindent=1em]  
\item   
$\mathcal{M}_{L(\log L)}^{i}(\vec{f})(x)   = \sup\limits_{Q\ni x}  \|f_{i}\|_{L(\log L) ,Q} \prod\limits_{j=1 \atop j\neq i}^{m} \dfrac{1}{|Q|}  \dint_{Q} |f_{j}(y_{j})| d y_{j} $,

\item  
$\mathcal{M}_{L(\log L)} (\vec{f})(x)   = \sup\limits_{Q\ni x}   \prod\limits_{j=1}^{m} \|f_{j}\|_{L(\log L) ,Q} $,
\end{enumerate}
\end{list}
where the supremum is taken over all the cubes $Q$ containing $x$.
\end{definition}

Obviously, if $r>1$, then the following pointwise estimates hold
\begin{align} \label{equ:mulmax-relation}
\mathcal{M}(\vec{f})(x)  \le C \mathcal{M}_{L(\log L)}^{i} (\vec{f})(x)   &\le C_{1} \mathcal{M}_{L(\log L)} (\vec{f})(x)   \le C_{2} \mathcal{M}_{r}(\vec{f})(x).
\end{align}
The first two inequalities in \eqref{equ:mulmax-relation} follows from \eqref{equ:holder-4} with $t=1$, that is
\begin{align*}
\frac{1}{|Q|}  \dint_{Q} |f_{j}(y_{j})| d y_{j}   &\le \|f_{j}\|_{L(\log L) ,Q},
\end{align*}
and the last one follows from the generalized Jensen's inequlity \cite[see Lemma 4.2 in ][]{lu2014multilinear}.

In addition, One can see that $\mathcal{M}_{L(\log L)} (\vec{f})$ is pointwise controlled by a multiple of
$ \prod\limits_{j=1}^{m} M^{2}(f_{j})(x)$, where $M^{2}=M\circ M$.

\begin{definition}[Lipschitz-type spaces]   \label{def.lip-space} \
\begin{enumerate}[ label=(\arabic*),itemindent=1em] 
\item  Let $0< \delta< 1$.  The space $\Lambda_{\delta}$ of the Lipschitz continuous functions with order $\delta$ is defined by
$$
 \Lambda_{\delta}(\mathbb{R}^{n}) = \{f: |f(x)-f(y)| \le C |x-y|^{\delta} \ \mbox{for a.e.}\ x,y\in \mathbb{R}^{n} \},
$$
where $f$  is the locally integrable  function on $\mathbb{R}^{n}$, and the   smallest  constant $C>0$  will be denoted Lipschitz norm by $\|f\|_{\Lambda_{\delta}}$.
\item  Let $0\le \delta< 1$.   The space $\mathbb{L}(\delta) $ is defined to be the set of all locally integrable  functions $f$, i.e., there exists a positive constant $C > 0$, such that
$$ \sup_{Q}\dfrac{1}{|Q|^{1+\delta/n} } \int_{Q} |f(y)-f_{Q}| dy<C,$$
where the supremum is taken over every cube $Q\subset \mathbb{R}^{n}$ and $f_{Q}=\frac{1}{|Q|} \int_{Q} f(z) dz$. The least constant $C$  will be
denoted by $\|f\|_{\mathbb{L}(\delta)}$.
\item \cite[see ][]{ramseyer2013lipschitz}\ Given $0 <\alpha < n$ and an exponent function $p(\cdot)\in \mathscr{P}_{1}(\R^{n})=\{1\}\cup\mathscr{P}(\R^{n})$. We say that a locally integrable function $f$ belongs to   $\mathbb{L}_{\alpha,p(\cdot)}=\mathbb{L}_{\alpha,p(\cdot)}(\mathbb{R}^{n})$ if there exists a constant $C$ such that
    $$\dfrac{1}{|B|^{\alpha/n}\|\chi_{B}\|_{p'(\cdot)}} \int_{B} |f(y)-f_{B}| dy< C$$
 for every ball $B\subset \mathbb{R}^{n}$, with $f_{B}=\frac{1}{|B|} \int_{B} f(z) dz$. The least constant $C$  will be
denoted by $\|f\|_{\mathbb{L}_{\alpha,p(\cdot)}}$.
\item  Let $r(\cdot)\in \mathscr{C}^{\log}(\R^{n}) \cap\mathscr{P}(\R^{n})$ such that $1<\beta\le r_{-}\le r(x)\le r_{+}<\frac{n\beta}{(n-\beta)^{+}}$, and set $\frac{\delta(x)}{n}=\frac{1}{\beta}-\frac{1}{r(x)}$. The space $\mathbb{L}(\delta(\cdot)) $ is defined by the set of the measurable functions $f$ such that (see \cite{ramseyer2013lipschitz} for more details )
$$\|f\|_{\mathbb{L}(\delta(\cdot)) } = \sup_{B}\dfrac{1}{|B|^{1/\beta}\|\chi_{B}\|_{r'(\cdot)}} \int_{B} |f(y)-f_{B}| dy<\infty.$$
\item (weighted Lipschitz integral spaces $\mathbb{L}_{w}(\delta)$, see \cite{pradolini2001class} or \cite{pradolini2001two})
 Let $w$ be a weight and $0\le \delta< 1$, we say that a locally integrable function $f$ belong to $\mathbb{L}_{w}(\delta)$ if there exists a positive constant $C$ such that the inequality
$$ \dfrac{\|w\chi_{_{B}}\|_{\infty}}{|B|^{1+\delta/n}} \int_{B} |f(y)-f_{B}| dy<C$$
holds for every ball $B\subset \mathbb{R}^{n}$. The least constant $C$ will be denoted by  $\|f\|_{\mathbb{L}_{w}(\delta) }$.
\end{enumerate}
\end{definition}

\begin{remark}  \label{rem.Lipschitz-def}
\begin{enumerate}[leftmargin=2em,label=(\roman*),itemindent=1.5em]  
\item  In (1) of   \cref{def.lip-space},  it   is well known that the space $\Lambda_{\delta}$ coincides with the space $\mathbb{L}(\delta)$  (see refs. \cite{harboure1997boundedness,pradolini2017characterization}).
\item  In (2) of   \cref{def.lip-space},  it is not difficult to see that, for   $\delta=0$, the space $\mathbb{L}(\delta) $ coincides with the space of bounded mean oscillation functions BMO (see  \cite{john1961functions}).
\item  In (4) of   \cref{def.lip-space}, denote $z^{+}$ by  (see  \cite{cabral2016extrapolation})
\begin{equation*}
z^{+}=
\left\{ \begin{aligned}
        z & \ \ \hbox{if } z>0\\
        0 &  \ \ \hbox{if }z\le  0
                          \end{aligned} \right..
                          \end{equation*}
In addition, when $r(x)$ is equal to a constant $r$, this space coincides with the space  $\mathbb{L}(n/\beta-n/r)$.
\item    In (5) of   \cref{def.lip-space},  it is not difficult to see that, for   $\delta=0$, the space $\mathbb{L}_{w}(\delta)$ coincides with one of the versions of weighted bounded mean oscillation spaces (see  \cite{muckenhoupt1976weighted}). Moreover, for the case $w\equiv 1$, the space $\mathbb{L}_{w}(\delta)$ is the known Lipschitz integral space for $0< \delta< 1$.
\end{enumerate}
\end{remark}

\begin{lemma}\cite{pradolini2017characterization} \label{lem:T8-condition}
Let  $p(\cdot)\in  \mathscr{C}^{\log}(\R^{n}) \cap\mathscr{P}(\R^{n})$ and $1<\beta\le p_{-}$. Then the functional
\begin{align}
 a(Q)   &= |Q|^{1/\beta-1} \|\chi_{Q}\|_{p'(\cdotp)}
\end{align}
satisfies the $T_{\infty}$ condition, that is, there exists a positive constant $C$ such that $a(Q') \le C a(Q) $ for each cube $Q$ and each cube $Q'\subset Q$.
\end{lemma}

\begin{lemma}\cite{li2005jhon} \label{lem:lip-equivalent}
Let  $1\le r<\infty$ and $a\in T_{\infty}$. Then
\begin{align*}
 \sup _{Q} \frac{1}{a(Q)} \left(\frac{1}{|Q|}\dint_{Q} \Big| f(x)-(f)_{Q}) \Big|^{r} dx \right)^{1/r}   &\approx \sup _{Q} \frac{1}{a(Q)} \frac{1}{|Q|}\dint_{Q} \Big| f(x)-(f)_{Q}) \Big| dx.
\end{align*}
\end{lemma}

The following inequalities are also necessary (see (2.16) in \cite{lerner2009new} or  Lemma 4.6 in \cite{lu2014multilinear}  or page 485 in \cite{garcia1985weighted}).

\begin{lemma}[Kolmogorov's inequality]  \label{lem:kolmogorov}
Let  $0<p<q<\infty $. Then there is a positive constant $C$ such that for any measurable funcction $f$ there has
\begin{align*}
 |Q|^{-1/p} \|f\|_{L^{p}(Q)}  &\le C |Q|^{-1/q} \|f\|_{L^{q,\infty}(Q)},
\end{align*}
where $L^{q,\infty}(Q) $ denotes   the weak space with norm $\|f\|_{L^{q,\infty}(Q)} = \sup\limits_{t>0} t|\{ x\in Q: |f(x)|>t\}|^{1/q} $.
\end{lemma}

\begin{lemma}\cite{pradolini2017characterization}\label{lem:lip-relation}
Let $p(\cdot)\in  \mathscr{C}^{\log}(\R^{n}) \cap\mathscr{P}(\R^{n})$ , $1<\beta\le p_{-}$  such that  $0\le \delta(\cdot)/n= 1/\beta - 1/p(\cdot) \le 1$ and $b \in \mathbb{L}(\delta(\cdot)) $.
 \begin{itemize}[ itemindent=1em]
\item[(1)] Then there is a positive constant $C$ there has
\begin{align}\label{equ:lip-relation-1}
 \sup _{Q} \dfrac{\| b-b_{Q}\|_{\exp L,Q}}{|Q|^{1/\beta-1} \|\chi_{Q}\|_{p'(\cdotp)}}    &\le C \|b\|_{\mathbb{L}(\delta(\cdot))}.
\end{align}

\item[(2)]  Then there is a positive constant $C$ such that for every $j\in \mathbb{N}$ there has
\begin{align}\label{equ:lip-relation-2}
  | b_{a^{k_{0}(j+1)}Q}-b_{a^{k_{0}}Q} |     &\le C j\|b\|_{\mathbb{L}(\delta(\cdot))}  |a^{k_{0}(j+1)}Q|^{1/\beta-1} \|\chi_{a^{k_{0}(j+1)}Q}\|_{p'(\cdotp)}.
\end{align}

\end{itemize}
\end{lemma}

\begin{lemma}\cite[see Lemma 4.11 in ][]{pradolini2017characterization} \label{lem.sharp-estimate-var-1}
Let  $0<\gamma<1 $, $r(\cdot), p(\cdot)\in  \mathscr{C}^{\log}(\R^{n}) \cap\mathscr{P}(\R^{n})$, and let $q(\cdot), \beta, \delta(\cdot)$  such that  $0\le \delta(\cdot)/n= 1/p(\cdot)- 1/q(\cdot)= 1/\beta - 1/r(\cdot) \le 1/n$. If $\|f\|_{q(\cdot)}<\infty$,   then there is a positive constant $C$ such that
\begin{align*}
 \|f\|_{q(\cdot)}  &\le C  \|f_{\delta(\cdot),\gamma}^{\sharp}\|_{p(\cdot)}.
\end{align*}
\end{lemma}

The following result is a generalization to the variable context of a pointwise estimate of commutators. 

\begin{lemma} \label{lem.sharp-estimate-var}
Let $m\ge 2$, $0<\gamma<\eta<1/m$, $ \epsilon_{0}<\epsilon \le 1$, $p(\cdot)\in  \mathscr{C}^{\log}(\R^{n}) \cap\mathscr{P}(\R^{n})$, $1<\beta\le p_{-}$ such that  $0\le \delta(\cdot)/n= 1/\beta - 1/p(\cdot) \le 1$ and $\vec{b}=(b_{1},b_{2},\dots,b_{m})\in \mathbb{L}(\delta(\cdot)) \times\mathbb{L}(\delta(\cdot)) \times\cdots\times \mathbb{L}(\delta(\cdot)) $. Then there exists a positive constant $C$ such that
\begin{align*}
  (T_{_{b_{j}}} (\vec{f}))_{\delta(\cdot),\gamma}^{\sharp} (x)   &\le C   \|b_{j}\|_{\mathbb{L}(\delta(\cdot))} \Big(M_{\eta}(T\vec{f})(x)+ \mathcal{M}_{L(\log L)}(\vec{f})(x) \Big)  ~(j=1,2,\dots,m).
\end{align*}
Furthermore,
\begin{align*}
  \big( T_{_{\Sigma \vec{b}}}(\vec{f}) \big)_{\delta(\cdot),\gamma}^{\sharp} (x)
  &\le C \sum_{j=1}^{m}  \|b_{j}\|_{\mathbb{L}(\delta(\cdot))} \Big(M_{\eta}\big(T(\vec{f})\big)(x)+ \mathcal{M}_{L(\log L)}(\vec{f})(x) \Big).
\end{align*}
\end{lemma}

\begin{proof}
Let $Q\subset \mathbb{R}^n$ and $x\in Q$. Due to the fact $\big| | a|^{\gamma} - | c|^{\gamma} \big| \le | a-c|^{\gamma}$ for $0<\gamma<1$, it is enough to show that, for some constant $C_{Q}$, there exists a positive constant $C$ such that
\begin{equation*}
  \dfrac{ \left(\dfrac{1}{|Q|}\dint_{Q} \Big| T_{_{b_{j}}}(\vec{f})(z) -C_{Q} \Big|^{\gamma} dz \right)^{1/\gamma}}{|Q|^{1/\beta-1} \|\chi_{Q}\|_{p'(\cdotp)} }\le C    \|b_{j}\|_{\mathbb{L}(\delta(\cdot))} \Big(M_{\eta}\big(T(\vec{f})\big)(x)+ \mathcal{M}_{L(\log L)}(\vec{f})(x) \Big).
\end{equation*}

For each $j$, we decompose $f_{j}=f_{j}^{0} + f_{j}^{\infty}$ with $f_{j}^{0}=f_{j} \chi_{a^{k_0}Q} $, where $a$ and $k_0$ are defined as in conditions \labelcref{equ:character-norm-est,equ.double-condition-2} respectively. Then
\begin{align*}
 \prod_{j=1}^{m} f_{j}(y_{j})   &=  \prod_{j=1}^{m} \Big( f_{j}^{0}(y_{j}) + f_{j}^{\infty}(y_{j}) \Big)
   =   \sum_{\alpha_{1},\dots,\alpha_{m}\in \{0,\infty\}}  f_{1}^{\alpha_{1}}(y_{1}) \cdots  f_{m}^{\alpha_{m}} (y_{m}) \\
  &=  \prod_{j=1}^{m} f_{j}^{0}(y_{j}) +  \sum_{(\alpha_{1},\dots,\alpha_{m})\in \ell} f_{1}^{\alpha_{1}}(y_{1}) \cdots  f_{m}^{\alpha_{m}} (y_{m}),
\end{align*}
where $\ell=\{ (\alpha_{1},\dots,\alpha_{m}): \hbox{there is at least one}~ \alpha_{j}\neq 0\}$.
 Let $\lambda$ be some positive constant to be chosen. It is easy to see that
\begin{align*}
 T_{_{b_{j}}}(\vec{f})(x)  &= (b_{j}(x)-\lambda) T(\vec{f})(x) -T(f_{1}^{0},\dots,(b_{j}-\lambda)f_{j}^{0} ,\dots, f_{m}^{0})(x) \\
 &\qquad - \sum_{(\alpha_{1},\dots,\alpha_{m})\in \ell} T(f_{1}^{\alpha_{1}},\dots,(b_{j}-\lambda)f_{j}^{\alpha_{j}} ,\dots, f_{m}^{\alpha_{m}})(x) .
\end{align*}

By taking $\lambda=(b_{j})_{a^{k_0}Q}$ and $C_{Q} = \sum\limits_{(\alpha_{1},\dots,\alpha_{m})\in \ell} \Big( T(f_{1}^{\alpha_{1}},\dots,(b_{j}-\lambda)f_{j}^{\alpha_{j}} ,\dots, f_{m}^{\alpha_{m}})  \Big)_{Q}$, we obtain that
\begin{equation*}
  \dfrac{ \left(\dfrac{1}{|Q|}\dint_{Q} \Big| T_{_{b_{j}}}(\vec{f})(z) -C_{Q} \Big|^{\gamma} dz \right)^{1/\gamma}}{|Q|^{1/\beta-1} \|\chi_{Q}\|_{p'(\cdotp)} }\le \dfrac{ C}{|Q|^{1/\beta-1} \|\chi_{Q}\|_{p'(\cdotp)} } \Big(I+II+{I\!I\!I} \Big),
\end{equation*}
where
\begin{align*}
 I  &= \left(\frac{1}{|Q|}\dint_{Q} \Big| (b_{j}(z)-\lambda) T(\vec{f})(z) \Big|^{\gamma} dz \right)^{1/\gamma}  \\
  II  &= \left(\frac{1}{|Q|}\dint_{Q} \Big| T(f_{1}^{0},\dots,(b_{j}-\lambda)f_{j}^{0} ,\dots, f_{m}^{0})(z) \Big|^{\gamma} dz \right)^{1/\gamma}  \\
{I\!I\!I}  &= \sum_{(\alpha_{1},\dots,\alpha_{m})\in \ell}  \bigg(\frac{1}{|Q|}\dint_{Q} \Big|T(f_{1}^{\alpha_{1}},\dots,(b_{j}-\lambda)f_{j}^{\alpha_{j}} ,\dots, f_{m}^{\alpha_{m}})(z)  \\
 &\hspace{6em} -  \Big( T(f_{1}^{\alpha_{1}},\dots,(b_{j}-\lambda)f_{j}^{\alpha_{j}} ,\dots, f_{m}^{\alpha_{m}})  \Big)_{Q} \Big|^{\gamma} dz \bigg)^{1/\gamma}   .
\end{align*}

Let us first estimate $I$. By taking $1<r<\eta/\gamma$ and using H\"{o}lder's inequality, we obtain that
\begin{align*}
 I  &\le C \left(\frac{1}{|Q|}\dint_{Q} \Big| (b_{j}(z)-(b_{j})_{a^{k_0}Q}) \Big|^{r'\gamma} dz \right)^{1/(r'\gamma)}  \left(\frac{1}{|Q|}\dint_{Q} \Big|  T(\vec{f})(z) \Big|^{r\gamma} dz \right)^{1/(r\gamma)} .
\end{align*}
It is know from \Cref{lem:T8-condition} that the functional $a(Q)=|Q|^{1/\beta-1} \|\chi_{Q}\|_{p'(\cdotp)}$ satisfies $T_{\infty}$ condition. Then, by \Cref{lem:lip-equivalent}, it can obtain that
\begin{align*}
 \dfrac{I}{|Q|^{1/\beta-1} \|\chi_{Q}\|_{p'(\cdotp)}}  &\le \dfrac{C \left(\dfrac{1}{|Q|}\dint_{Q} \Big| \Big(b_{j}(z)-(b_{j})_{a^{k_0}Q} \Big) \Big|^{r'\gamma} dz \right)^{1/(r'\gamma)} }{|Q|^{1/\beta-1} \|\chi_{Q}\|_{p'(\cdotp)}} \left(\frac{1}{|Q|}\dint_{Q} \Big|  T(\vec{f})(z) \Big|^{r\gamma} dz \right)^{1/(r\gamma)} \\
  &\le C \|b_{j}\|_{\mathbb{L}(\delta(\cdot))}  M_{r\gamma}\big(T(\vec{f})\big)(x)   \\
  &\le C \|b_{j}\|_{\mathbb{L}(\delta(\cdot))}  M_{\eta}\big(T(\vec{f})\big)(x).
\end{align*}

 To estimate $II$, note that $0<\gamma<1/m$, by using $L^{1}  \times L^{1}  \times\cdots\times L^{1}$   to $L^{1/m,\infty}$ boundedness of $T$, Kolmogorov's inequality (\cref{lem:kolmogorov}), H\"{o}lder's inequality (\ref{equ:holder-4}) with $t=1$  and condition \eqref{equ:mulmax-relation}, we have
 \begin{align*}
  II  &= \left(\frac{1}{|Q|}\dint_{Q} \Big| T(f_{1}^{0},\dots,(b_{j}-\lambda)f_{j}^{0} ,\dots, f_{m}^{0}) \Big|^{\gamma} dz \right)^{1/\gamma}  \\
   &\le \frac{C}{|Q|^{m}}\Big\| T(f_{1}^{0},\dots,(b_{j}-\lambda)f_{j}^{0} ,\dots, f_{m}^{0})(z) \Big\|_{L^{1/m,\infty}(Q)}  \\
 &\le \frac{C}{|Q|^{m}} \| (b_{j}-(b_{j})_{a^{k_0}Q})f_{j}^{0} \|_{L^{1}(\mathbb{R}^{n})} \prod_{k=1\atop k\neq j}^{m}\|f_{k}^{0}\|_{L^{1}(\mathbb{R}^{n})}  \\
 &\le \frac{C}{|Q|} \dint_{Q} | b_{j}(z)-(b_{j})_{a^{k_0}Q}| |f_{j}(z) | dz \prod_{k=1\atop k\neq j}^{m} \frac{1}{|Q|}  \dint_{Q} |f_{k}(z) | dz \\
 &\le C \| b_{j}-(b_{j})_{a^{k_0}Q}\|_{\exp L,a^{k_0}Q}  \|f_{j}\|_{L(\log L),a^{k_0}Q}  \prod_{k=1\atop k\neq j}^{m} \frac{1}{|a^{k_0}Q|}  \dint_{a^{k_0}Q} |f_{k}(z) | dz \\
  &\le C \| b_{j}-(b_{j})_{a^{k_0}Q}\|_{\exp L,a^{k_0}Q}  \mathcal{M}_{L(\log L)}^{j}(\vec{f}) (x)   \\
    &\le C \| b_{j}-(b_{j})_{a^{k_0}Q}\|_{\exp L,a^{k_0}Q}  \mathcal{M}_{L(\log L)}(\vec{f}) (x) .
\end{align*}

Thus, by \cref{lem:T8-condition,lem:lip-relation}  and doubling condition  implied in conditions \labelcref{equ:character-norm-est,equ.double-condition-2}, we get that
\begin{align*}
 \dfrac{II}{|Q|^{1/\beta-1} \|\chi_{Q}\|_{p'(\cdotp)}}  &\le \dfrac{C \| b_{j}-(b_{j})_{a^{k_0}Q}\|_{\exp L,a^{k_0}Q}}{|Q|^{1/\beta-1} \|\chi_{Q}\|_{p'(\cdotp)}}  \mathcal{M}_{L(\log L)}(\vec{f}) (x) \\
  &\le C\dfrac{ \| b_{j}-(b_{j})_{a^{k_0}Q}\|_{\exp L,a^{k_0}Q}}{|a^{k_0}Q|^{1/\beta-1} \|\chi_{a^{k_0}Q}\|_{p'(\cdotp)}}  \mathcal{M}_{L(\log L)}(\vec{f}) (x) \\
  &\le C \|b_{j}\|_{\mathbb{L}(\delta(\cdot))}  \mathcal{M}_{L(\log L)}(\vec{f}) (x).
\end{align*}

To estimate ${I\!I\!I}$, we consider first the case when $\alpha_{1}=\cdots =\alpha_{m}=\infty$. For any $z\in Q$, there has
\begin{align*}
&  \Big|T(f_{1}^{\infty},\dots,(b_{j}-(b_{j})_{a^{k_0}Q})f_{j}^{\infty} ,\dots, f_{m}^{\infty})(z)   -   \Big( T(f_{1}^{\infty},\dots,(b_{j}-(b_{j})_{a^{k_0}Q})f_{j}^{\infty} ,\dots, f_{m}^{\infty})  \Big)_{Q}   \Big|  \\
&\le \frac{1}{|Q|} \int_{Q}  \Big(\int_{(\mathbb{R}^{n}\setminus  a^{k_0}Q)^{m}}  |b_{j}(y_{j})-(b_{j})_{a^{k_0}Q}| |K(z,\vec{y})-K(w,\vec{y})| \prod_{i=1}^{m}| f_{i}^{\infty}(y_{i})|  d\vec{y} \Big) dw \\
&\le \frac{1}{|Q|} \int_{Q}  \Big( \sum_{k=1}^{\infty}\int_{(\mathcal {Q}_{k})^{m}}  |b_{j}(y_{j})-(b_{j})_{a^{k_0}Q}| |K(z,\vec{y})-K(w,\vec{y})| \prod_{i=1}^{m}| f_{i}^{\infty}(y_{i})|  d\vec{y} \Big) dw \\
&\le  \frac{1}{|Q|} \sum_{k=1}^{\infty} \Big( \int_{Q}\int_{(\mathcal {Q}_{k})^{m}}  |b_{j}(y_{j})-(b_{j})_{a^{k_0}Q}| |K(z,\vec{y})-K(w,\vec{y})| \prod_{i=1}^{m}| f_{i}^{\infty}(y_{i})|  d\vec{y} \Big) dw ,
\end{align*}
where $\mathcal {Q}_{k}= ( a^{k_0(k+1)}Q) \setminus (a^{k_0k}Q) $ for $k=1,2,\dots$. Note that, for $w,z\in Q$ and any $(y_{1},\dots,y_{m}) \in (\mathcal {Q}_{k})^{m}$, there has
\begin{align*}
a^{k_0k}l(Q)\le |z-y_{i}|< a^{k_0(k+1)}l(Q) & ~\hbox{and} ~ |z-w|\le  a^{k_0} l(Q),
\end{align*}
and applying \eqref{equ:CZK-2}, we have
\begin{align} \label{equ:CZK-2-1}
|K(z,\vec{y})-K(w,\vec{y})| &\le \dfrac{A|z-w|^{\epsilon}}{(|z-y_{1}|+\cdots+|z-y_{m}|)^{mn+\epsilon}}  
\le \dfrac{Ca^{-k_0k\epsilon}}{|a^{k_0k} Q |^{m}} .
\end{align}
Then
\begin{align} \label{equ:III-infty-1}
 & \Big|T(f_{1}^{\infty},\dots,(b_{j}-(b_{j})_{Q^{*}})f_{j}^{\infty} ,\dots, f_{m}^{\infty})(z)   -  \Big( T(f_{1}^{\infty},\dots,(b_{j}-(b_{j})_{a^{k_0}Q})f_{j}^{\infty} ,\dots, f_{m}^{\infty})  \Big)_{Q} \Big|    \notag \\
&\le  \frac{1}{|Q|} \sum_{k=1}^{\infty}\int_{Q} \Big( \int_{(\mathcal {Q}_{k})^{m}}  |b_{j}(y_{j})-(b_{j})_{a^{k_0}Q}| |K(z,\vec{y})-K(w,\vec{y})| \prod_{i=1}^{m}| f_{i}^{\infty}(y_{i})|  d\vec{y} \Big) dw    \notag \\
&\le  \frac{C}{|Q|} \sum_{k=1}^{\infty}  a^{-k_0k\epsilon}\int_{Q} \Big( \int_{(\mathcal {Q}_{k})^{m}}  |b_{j}(y_{j})-(b_{j})_{a^{k_0}Q}| \dfrac{1}{|a^{k_0k} Q |^{m}} \prod_{i=1}^{m}| f_{i}^{\infty}(y_{i})|  d\vec{y} \Big) dw    \notag \\
&\le C \sum_{k=1}^{\infty}  a^{-k_0k\epsilon} \int_{(\mathcal {Q}_{k})^{m}}  |b_{j}(y_{j})-(b_{j})_{a^{k_0}Q}| \dfrac{1}{|a^{k_0k} Q |^{m}} \prod_{i=1}^{m}| f_{i}^{\infty}(y_{i})|  d\vec{y}     \\
&\le C\sum_{k=1}^{\infty}  a^{-k_0k\epsilon} \Big( \frac{1}{|a^{k_0(k+1)}Q | } \int_{a^{k_0(k+1)}Q}  |b_{j}(y_{j})-(b_{j})_{a^{k_0}Q}| | f_{j}(y_{j})|  dy_{j} \Big)    \notag \\
 &\hspace{8em} \times \Big(\prod_{i=1\atop i\neq j}^{m}\frac{1}{|a^{k_0(k+1)}Q | } \int_{a^{k_0(k+1)}Q}  | f_{i} (y_{i})|  dy_{i} \Big).         \notag
 \end{align}

Therefore, by \cref{lem:lip-relation} and   the  generalized  H\"{o}lder's inequality in Orlicz space (see \Cref{equ:holder-4} with $t=1$), we have
\begin{align} \label{equ:III-infty-2}
 &\;   \frac{1}{|a^{k_0(k+1)}Q | } \int_{a^{k_0(k+1)}Q}  |b_{j}(y_{j})-(b_{j})_{a^{k_0}Q}| | f_{j}(y_{j})|  dy_{j}    \nonumber \\
 &\; = \frac{1}{|a^{k_0(k+1)}Q | } \int_{a^{k_0(k+1)}Q}  |b_{j}(y_{j})-(b_{j})_{a^{k_0(k+1)}Q}+(b_{j})_{a^{k_0(k+1)}Q} -(b_{j})_{a^{k_0}Q}| | f_{j}(y_{j})|  dy_{j}    \nonumber\\
 &\;  \le \frac{1}{|a^{k_0(k+1)}Q | } \int_{a^{k_0(k+1)}Q}  |b_{j}(y_{j})-(b_{j})_{a^{k_0(k+1)}Q}| | f_{j}(y_{j})|  dy_{j}     \nonumber \\
 &\; \hspace{12em} +\frac{|(b_{j})_{a^{k_0(k+1)}Q} -(b_{j})_{a^{k_0}Q}|}{|a^{k_0(k+1)}Q | } \int_{a^{k_0(k+1)}Q}   | f_{j}(y_{j})|  dy_{j}    \\
 &\; \le  C\| b_{j}-(b_{j})_{a^{k_0(k+1)}Q}\|_{\exp L,a^{k_0(k+1)}Q}   \|f_{j}\|_{L(\log L),a^{k_0(k+1)}Q}      \nonumber\\
 &\; \hspace{12em}   + |(b_{j})_{a^{k_0(k+1)}Q}-(b_{j})_{a^{k_0}Q}|   \|f_{j}\|_{L(\log L) ,a^{k_0(k+1)}Q}    \nonumber\\
 &\; \le  C \Big( \frac{\| b_{j}-(b_{j})_{a^{k_0(k+1)}Q}\|_{\exp L,a^{k_0(k+1)}Q}}{|a^{k_0(k+1)}Q|^{1/\beta-1} \|\chi_{a^{k_0(k+1)}Q}\|_{p'(\cdotp)} }  +   k  \|b_{j}\|_{\mathbb{L}(\delta(\cdot))} \Big)      \nonumber \\
 &\; \hspace{12em} \times   |a^{k_0(k+1)}Q|^{1/\beta-1} \|\chi_{a^{k_0(k+1)}Q}\|_{p'(\cdotp)}\|f_{j}\|_{L(\log L),a^{k_0(k+1)}Q}    \nonumber\\
 &\;  \le  C  (k+1)  \|b_{j}\|_{\mathbb{L}(\delta(\cdot))}   |a^{k_0(k+1)}Q|^{1/\beta-1} \|\chi_{a^{k_0(k+1)}Q}\|_{p'(\cdotp)}\|f_{j}\|_{L(\log L),a^{k_0(k+1)}Q}.     \nonumber
\end{align}

Since $0<\gamma<1$, by H\"{o}lder's inequality, \Cref{lem:lip-relation}, \labelcref{equ:character-norm-est,equ:mulmax-relation,equ:III-infty-1,equ:III-infty-2}, and the fact that $ 1/k_{0}=\epsilon_{0}<\epsilon \le 1$,  we have
\begin{align} \label{equ:III-infty-3}
& \dfrac{{I\!I\!I}_{\infty}}{|Q|^{1/\beta-1} \|\chi_{Q}\|_{p'(\cdotp)}}= \dfrac{1}{|Q|^{1/\beta-1} \|\chi_{Q}\|_{p'(\cdotp)}}  \bigg(\frac{1}{|Q|}\dint_{Q} \Big|T(f_{1}^{\infty},\dots,(b_{j}-(b_{j})_{Q^{*}})f_{j}^{\infty} ,\dots, f_{m}^{\infty})(z)      \notag\\
 &\hspace{6em} -   \Big( T(f_{1}^{\infty},\dots,(b_{j}-(b_{j})_{a^{k_0}Q})f_{j}^{\infty} ,\dots, f_{m}^{\infty})  \Big)_{Q}  \Big|^{\gamma} dz \bigg)^{1/\gamma}               \notag\\
  &\le \dfrac{1}{|Q|^{1/\beta-1} \|\chi_{Q}\|_{p'(\cdotp)}}  \bigg(\frac{C}{|Q|}\dint_{Q} \Big|T(f_{1}^{\infty},\dots,(b_{j}-(b_{j})_{Q^{*}})f_{j}^{\infty} ,\dots, f_{m}^{\infty})(z)   \notag\\
 & \hspace{6em} -   \Big( T(f_{1}^{\infty},\dots,(b_{j}-(b_{j})_{a^{k_0}Q})f_{j}^{\infty} ,\dots, f_{m}^{\infty})  \Big)_{Q} \Big|  dz \bigg)   \notag\\
 &\le \bigg(\frac{C}{|Q|}\dint_{Q} \sum_{k=1}^{\infty}  a^{-k_0k\epsilon} \Big( \frac{1}{|a^{k_0(k+1)}Q | } \int_{a^{k_0(k+1)}Q}  |b_{j}(y_{j})-(b_{j})_{a^{k_0}Q}| | f_{j}(y_{j})|  dy_{j} \Big)   \notag\\
 &\hspace{8em} \times \Big(\prod_{i=1\atop i\neq j}^{m}\frac{1}{|a^{k_0(k+1)}Q | } \int_{a^{k_0(k+1)}Q}  | f_{i} (y_{i})|  dy_{i} \Big)  dz \bigg)   \dfrac{1}{|Q|^{1/\beta-1} \|\chi_{Q}\|_{p'(\cdotp)}}   \notag\\
&\le \dfrac{C \|b_{j}\|_{\mathbb{L}(\delta(\cdot))} }{|Q|^{1/\beta-1} \|\chi_{Q}\|_{p'(\cdotp)}}\bigg(  \sum_{k=1}^{\infty}  a^{-k_0k\epsilon} (k+1) |a^{k_0(k+1)}Q|^{1/\beta-1} \|\chi_{a^{k_0(k+1)}Q}\|_{p'(\cdotp)}   \\
 &\hspace{8em} \times \|f_{j}\|_{L(\log L),a^{k_0(k+1)}Q}\Big(\prod_{i=1\atop i\neq j}^{m}\frac{1}{|a^{k_0(k+1)}Q | } \int_{a^{k_0(k+1)}Q}  | f_{i} (y_{i})|  dy_{i} \Big)   \bigg)    \notag\\
&\le C \|b_{j}\|_{\mathbb{L}(\delta(\cdot))}  \mathcal{M}_{L(\log L)}^{j}(\vec{f}) (x) \Big(  \sum_{k=1}^{\infty}  a^{-k_0k\epsilon} (k+1) \dfrac{ |a^{k_0(k+1)}Q|^{1/\beta-1} \|\chi_{a^{k_0(k+1)}Q}\|_{p'(\cdot)} }{|Q|^{1/\beta-1} \|\chi_{Q}\|_{p'(\cdotp)}}   \Big)   \notag\\
&\le C \|b_{j}\|_{\mathbb{L}(\delta(\cdot))}  \mathcal{M}_{L(\log L)}(\vec{f}) (x) \Big(  \sum_{k=1}^{\infty}  a^{-k_0k\epsilon} (k+1) a^{k_0kn(1/\beta-1)}\dfrac{ \|\chi_{a^{k_0(k+1)}Q}\|_{p'(\cdot)} }{  \|\chi_{Q}\|_{p'(\cdotp)}}   \Big)   \notag\\
&\le C \|b_{j}\|_{\mathbb{L}(\delta(\cdot))}  \mathcal{M}_{L(\log L)}(\vec{f}) (x) \Big(  \sum_{k=1}^{\infty}  a^{-k_0k\epsilon} (k+1) a^{k_0kn(1/\beta-1)}\dfrac{ a^{k_0k(n-n/\beta+1)} }{  2^{k}}   \Big)  \notag \\
&\le C \|b_{j}\|_{\mathbb{L}(\delta(\cdot))}  \mathcal{M}_{L(\log L)}(\vec{f}) (x) \sum_{k=1}^{\infty}   (k+1)  \Big( \dfrac{ a^{k_0(1-\epsilon)} }{  2}   \Big)^{k}   \notag\\
&\le C \|b_{j}\|_{\mathbb{L}(\delta(\cdot))}  \mathcal{M}_{L(\log L)}(\vec{f}) (x) .     \notag
 \end{align}

From estimates \labelcref{equ:III-infty-1,equ:III-infty-2,equ:III-infty-3}, there holds the following inequality, which will be used later,
 \begin{align} \label{equ:III-infty-4}
\begin{aligned}
 &\dfrac{1}{|Q|^{1/\beta-1} \|\chi_{Q}\|_{p'(\cdotp)}}\bigg(  \sum_{k=1}^{\infty}  a^{-k_0k\epsilon} \Big( \int_{(a^{k_0(k+1)}Q)^{m}} \frac{|b_{j}(y_{j})-(b_{j})_{a^{k_0}Q}| | f_{j}(y_{j})|}{|a^{k_0(k+1)}Q |^{m} }  \prod_{i=1\atop i\neq j}^{m}   | f_{i} (y_{i})|  d\vec{y} \Big)  \bigg)   \\
&\le C \|b_{j}\|_{\mathbb{L}(\delta(\cdot))}  \mathcal{M}_{L(\log L)}(\vec{f}) (x) .
\end{aligned}
 \end{align}

 Now, for $(\alpha_{1},\dots,\alpha_{m})\in \ell$, let us consider the terms $ {I\!I\!I}_{\alpha_{1},\dots,\alpha_{m}}$ such that at least one
 $\alpha_{j}=0$ and  one  $\alpha_{i}=\infty$. Without loss of generality, we assume that  $ \alpha_{1}=\cdots=\alpha_{l}=0$ and $\alpha_{l+1}=\cdots=\alpha_{m}=\infty$ with $1\le l<m$. For any $z\in Q$, set $\mathcal {Q}_{k}= ( a^{k_0(k+1)}Q) \setminus (a^{k_0k}Q) $ as above, when $l+1\le j \le m$, applying \eqref{equ:CZK-2-1}, we obtain that
\begin{align*}
&  \Big|T(f_{1}^{\alpha_{1}},\dots,(b_{j}-(b_{j})_{a^{k_0}Q})f_{j}^{\alpha_{j}} ,\dots, f_{m}^{\alpha_{m}})(z)   -   \Big( T(f_{1}^{\alpha_{1}},\dots,(b_{j}-(b_{j})_{a^{k_0}Q})f_{j}^{\alpha_{j}} ,\dots, f_{m}^{\alpha_{m}})  \Big)_{Q}   \Big|  \\
&\le \frac{1}{|Q|} \int_{Q}  \Big(\int_{(\mathbb{R}^{n})^{m}}  |b_{j}(y_{j})-(b_{j})_{a^{k_0}Q}| |K(z,\vec{y})-K(w,\vec{y})|  \prod_{i=1}^{m}| f_{i}^{\alpha_{i}}(y_{i})|  d\vec{y} \Big) dw \\
&\le \frac{1}{|Q|} \int_{Q}  \Big( \int_{( a^{k_0}Q)^{l}\times (\mathbb{R}^{n}\setminus  a^{k_0}Q)^{m-l}}  |b_{j}(y_{j})-(b_{j})_{a^{k_0}Q}| |K(z,\vec{y})-K(w,\vec{y})|    \\
& \hspace{10em} \times \prod_{i=1}^{l}| f_{i}^{0}(y_{i})| \prod_{i=l+1}^{m}| f_{i}^{\infty}(y_{i})|  d\vec{y} \Big)dw \\
&\le \frac{1}{|Q|} \int_{Q} \bigg( \int_{( a^{k_0}Q)^{l}}\prod_{i=1}^{l}| f_{i}^{0}(y_{i})|  \sum_{k=1}^{\infty}\int_{(\mathcal {Q}_{k})^{m-l}} |b_{j}(y_{j})-(b_{j})_{a^{k_0}Q}|    \\
& \hspace{10em} \times |K(z,\vec{y})-K(w,\vec{y})|\prod_{i=l+1}^{m}| f_{i}^{\infty}(y_{i})|  d\vec{y} \bigg) dw \\
&\le \frac{C}{|Q|} \int_{Q} \bigg( \int_{( a^{k_0}Q)^{l}}\prod_{i=1}^{l}| f_{i}^{0}(y_{i})|  \sum_{k=1}^{\infty} a^{-k_0k\epsilon}\int_{(\mathcal {Q}_{k})^{m-l}} \frac{|b_{j}(y_{j})-(b_{j})_{a^{k_0}Q}| }{|a^{k_0k} Q |^{m}}   \prod_{i=l+1}^{m}| f_{i}^{\infty}(y_{i})|  d\vec{y} \bigg) dw   \\
&\le C \sum_{k=1}^{\infty} a^{-k_0k\epsilon}\bigg( \int_{(a^{k_0(k+1)}Q)^{m}} \frac{|b_{j}(y_{j})-(b_{j})_{a^{k_0}Q}| }{|a^{k_0(k+1)}Q |^{m}}   \prod_{i=1}^{m}| f_{i} (y_{i})|  d\vec{y} \bigg)   \\
&\le C \sum_{k=1}^{\infty} a^{-k_0k\epsilon}\bigg( \int_{(a^{k_0(k+1)}Q)^{m}} \frac{|b_{j}(y_{j})-(b_{j})_{a^{k_0}Q}|  |f_{j}(y_{j})|}{|a^{k_0(k+1)}Q |^{m}}   \prod_{i=1 \atop i\neq j}^{m}| f_{i} (y_{i})|  d\vec{y} \bigg).
\end{align*}

When $1\le j \le l$,  similar to the above, we have that 
\begin{align*}
&   \Big|T(f_{1}^{\alpha_{1}},\dots,(b_{j}-(b_{j})_{a^{k_0}Q})f_{j}^{\alpha_{j}} ,\dots, f_{m}^{\alpha_{m}})(z)   -   \Big( T(f_{1}^{\alpha_{1}},\dots,(b_{j}-(b_{j})_{a^{k_0}Q})f_{j}^{\alpha_{j}} ,\dots, f_{m}^{\alpha_{m}})  \Big)_{Q}   \Big|  \\
&\le C \sum_{k=1}^{\infty} a^{-k_0k\epsilon}\bigg( \int_{(a^{k_0(k+1)}Q)^{m}} \frac{|b_{j}(y_{j})-(b_{j})_{a^{k_0}Q}|  |f_{j}(y_{j})|}{|a^{k_0(k+1)}Q |^{m}}   \prod_{i=1 \atop i\neq j}^{m}| f_{i} (y_{i})|  d\vec{y} \bigg).
\end{align*}

Since $0<\gamma<1$, by H\"{o}lder's inequality and \eqref{equ:III-infty-4}, we have
\begin{align*}
& \dfrac{{I\!I\!I}_{\alpha_{1},\dots,\alpha_{m}}}{|Q|^{1/\beta-1} \|\chi_{Q}\|_{p'(\cdotp)}}= \dfrac{1}{|Q|^{1/\beta-1} \|\chi_{Q}\|_{p'(\cdotp)}}  \bigg(\frac{1}{|Q|}\dint_{Q} \Big|T(f_{1}^{\alpha_{1}},\dots,(b_{j}-(b_{j})_{a^{k_0}Q})f_{j}^{\alpha_{j}} ,\dots, f_{m}^{\alpha_{m}})(z)   \\
 &\hspace{6em}  -   \Big( T(f_{1}^{\alpha_{1}},\dots,(b_{j}-(b_{j})_{a^{k_0}Q})f_{j}^{\alpha_{j}} ,\dots, f_{m}^{\alpha_{m}})  \Big)_{Q}  \Big|^{\gamma} dz \bigg)^{1/\gamma}  \\
  &\le \dfrac{1}{|Q|^{1/\beta-1} \|\chi_{Q}\|_{p'(\cdotp)}}  \bigg(\frac{C}{|Q|}\dint_{Q} \Big|T(f_{1}^{\alpha_{1}},\dots,(b_{j}-(b_{j})_{a^{k_0}Q})f_{j}^{\alpha_{j}} ,\dots, f_{m}^{\alpha_{m}})(z)  \\
 & \hspace{6em}  -   \Big( T(f_{1}^{\alpha_{1}},\dots,(b_{j}-(b_{j})_{a^{k_0}Q})f_{j}^{\alpha_{j}} ,\dots, f_{m}^{\alpha_{m}})  \Big)_{Q} \Big|  dz \bigg)   \\
&\le   \dfrac{C}{|Q|^{1/\beta-1} \|\chi_{Q}\|_{p'(\cdotp)}}\bigg(  \sum_{k=1}^{\infty}  a^{-k_0k\epsilon} \Big( \int_{(a^{k_0(k+1)}Q)^{m}} \frac{|b_{j}(y_{j})-(b_{j})_{a^{k_0}Q}| | f_{j}(y_{j})|}{|a^{k_0(k+1)}Q |^{m} }  \prod_{i=1\atop i\neq j}^{m}   | f_{i} (y_{i})|  d\vec{y} \Big)  \bigg)                  \\
&\le C \|b_{j}\|_{\mathbb{L}(\delta(\cdot))}  \mathcal{M}_{L(\log L)}(\vec{f}) (x) .
 \end{align*}

Combining the above estimates we get the desired result. The proof is completed.
\end{proof}




\section{Main results and their proofs}

 Our main result can be stated as follows.

 The following theorem gives a characterization of the spaces $\mathbb{L}(\delta)$ in terms of the
boundedness of $T_{_{\Sigma \vec{b}}}$ between variable Lebesgue spaces.

\begin{theorem}\label{thm.1-0}
Suppose that $0< \delta<1$, $p_{_{1}}(\cdot),p_{_{2}}(\cdot),\dots,p_{_{m}}(\cdot)\in \mathscr{C}^{\log}(\R^{n}) \cap\mathscr{P}(\R^{n})$ satisfy $\frac{1}{p(x)}= \frac{1}{p_{_{1}}(x)} + \frac{1}{p_{_{2}}(x)}+\cdots+\frac{1}{p_{_{m}}(x)}$, and  $\frac{mn}{\delta+n} <(p_{j})_{-}<(p_{j})_{+}< \frac{mn}{\delta} ~(j=1,2,\dots,m)$. Define the variable exponent $q(\cdot)$ by
$$\frac{1}{q(x)}=\frac{1}{p(x)}-\frac{\delta}{n}.$$
Then $\vec{b}=(b_{1},b_{2},\dots,b_{m})\in \mathbb{L}(\delta) \times\mathbb{L}(\delta) \times\cdots\times \mathbb{L}(\delta) $ if and only if $T_{_{b_{j}}}: L^{p_1(\cdot)}(\mathbb{R}^{n}) \times L^{p_2(\cdot)}(\mathbb{R}^{n}) \times\cdots\times L^{p_m(\cdot)}(\mathbb{R}^{n}) \to L^{q(\cdot)}(\mathbb{R}^{n}) ~(j=1,2,\dots,m)$.

\end{theorem}
%
%
%

\begin{proof}
Without loss of generality, we only consider the case that $m=2$. Actually, similar procedure work for all $m\in \mathbb{N}$.

$\xLongleftarrow{\ \ \ \ }$:\  We first prove that the condition is sufficient.
Assume that $T_{_{b_{j}}} ~(j=1,2)$ maps $L^{p_1(\cdot)}(\mathbb{R}^{n}) \times L^{p_2(\cdot)}(\mathbb{R}^{n})$ into $ L^{q(\cdot)}(\mathbb{R}^{n})$. Since $1/K(\vec{z})$ is infinitely differentiable in almost every $\vec{z}=(z_1,z_2)\in \mathbb{R}^{n}\times \mathbb{R}^{n}$. Consequently, note that the homogeneity of $K$, we may choose $\vec{z}_{_{0}}=(z_{_{10}},z_{_{20}})\in \mathbb{R}^{n}\backslash \{0\} \times \mathbb{R}^{n} \backslash \{0\} $ 
and $\epsilon>0$ satisfy that  $|z_{_{j}}-z_{_{j0}}|< \epsilon \sqrt{n} ~(j=1,2)$ which implies that $|\vec{z}-\vec{z}_{_{0}}|< \epsilon \sqrt{2n}$. Then  $1/K(\vec{z})$,  on the ball $B(\vec{z}_{_{0}},\epsilon \sqrt{2n})\subset \mathbb{R}^{2n} $ ,  can be expanded as an absolutely convergent Fourier series
$$\dfrac{1}{K(\vec{z})}=\dfrac{1}{K(z_1,z_2)}= \sum_{k=0}^{\infty} a_{k} \mathrm{e}^{\mathi\vec{v}_{k}\cdot \vec{z}} = \sum_{k=0}^{\infty} a_{k} \mathrm{e}^{\mathi(v_{1k},v_{2k})\cdot(z_1,z_2)}, $$
where the individual vectors $\vec{v}_{k} = (v_{1k},v_{2k}) \in \mathbb{R}^{n} \times \mathbb{R}^{n}$ do not play any significant role in the proof.

Set  ${\vec{z}}^{^{*}} =(z_{1}^{*},z_{2}^{*})$. If $z_{j}^{*} =\epsilon^{-1} z_{j0}$, then $|z_j-z_{j}^{*}|<  \sqrt{n}$ implies that  $|\epsilon z_j-z_{j0}|< \epsilon \sqrt{n}~(j=1,2)$. Thus, by the homogeneity of $K$, for all $\vec{z}\in B({\vec{z}}^{^{*}},\sqrt{2n})\subset \mathbb{R}^{2n} $, we can obtain that
$$\dfrac{1}{K(\vec{z})}=  \dfrac{\epsilon^{-2n}}{K(\epsilon\vec{z})}=\epsilon^{-2n} \sum_{k=0}^{\infty} a_{k} \mathrm{e}^{\mathi\epsilon\vec{v}_{k}\cdot \vec{z}}. $$

Let $Q=Q(x_0,l)$ be an arbitrary cube in $\mathbb{R}^{n}$ with sides parallel to the coordinate axes, diameter $l$ and center $x_0$, and set $y_{j0} = x_0-l z_{j}^*$ and $Q_{j}^{*} = Q(y_{j0},l) \subset \mathbb{R}^{n} ~(j=1,2)$. By taking $x\in Q, y_j\in Q_{j}^{*}~(j=1,2)$, we have that

\begin{align*}
\left|\dfrac{x-y_j}{l}- z_{j}^*\right|  &= \left|\dfrac{x-x_0+x_0-y_{j0}+y_{j0}-y_j}{l}- z_{j}^*\right|    \\
& \le \left|\dfrac{x-x_0}{l}\right| +\left|\dfrac{y_j-y_{j0}}{l}\right|  \le \sqrt{n} .  
\end{align*}
So we have $|(\frac{x-y_1}{l},\frac{x-y_2}{l})-{\vec{z}}^{^{*}}|= |(\frac{x-y_1}{l},\frac{x-y_2}{l})-(z_{1}^{*},z_{2}^{*})|\le \sqrt{2n} $, that is, $(\frac{x-y_1}{l},\frac{x-y_2}{l})\in B({\vec{z}}^{^{*}},\sqrt{2n})$, which means that $(x-y_1,x-y_2)$ is bounded away from the singularity of $K$.

Without loss of generality, let $s_1(x)=\sgn(b_1(x)-(b_1)_{Q_{1}^{*}})$, then
\begin{align} \label{equ:thm1-1}
\dint_Q &|b_{1}(x)-(b_1)_{Q_{1}^{*}}| dx= \dint_Q (b_{1}(x)-(b_1)_{Q_{1}^{*}})s_1(x) dx     \notag\\
 &=  \dint_Q \Big(\dfrac{1}{|Q_{1}^{*}|}\dint_{Q_{1}^{*}}(b_{1}(x)-b_1(y_1)) dy_1 \Big)\Big(\dfrac{1}{|Q_{2}^{*}|}\dint_{Q_{2}^{*}} \chi_{_{Q_{2}^{*}}}(y_2)  dy_2 \Big)s_1(x) dx               \notag\\
 &= \dfrac{1}{|Q_{1}^{*}|} \dfrac{1}{|Q_{2}^{*}|} \dint_Q  \dint_{Q_{1}^{*}}\dint_{Q_{2}^{*}} (b_{1}(x)-b_1(y_1))  s_1(x) dy_2 dy_1  dx          \notag  \\
 &= l^{-2n} \dint_{\mathbb{R}^{n}}  \dint_{\mathbb{R}^{n}}\dint_{\mathbb{R}^{n}} (b_{1}(x)-b_1(y_1)) \chi_{_{Q_{1}^{*}}}(y_1) \chi_{_{Q_{2}^{*}}}(y_2) \chi_{_{Q}}(x) s_1(x) dy_1 dy_2  dx           \notag \\
 &= l^{-2n} \dint_{\mathbb{R}^{n}}  \dint_{\mathbb{R}^{n}}\dint_{\mathbb{R}^{n}} (b_{1}(x)-b_1(y_1)) \dfrac{l^{2n} K(x-y_1,x-y_2)}{K(\frac{x-y_1}{l},\frac{x-y_2}{l})} \\
 &\hspace{8em}   \times \chi_{_{Q_{1}^{*}}}(y_1) \chi_{_{Q_{2}^{*}}}(y_2) \chi_{_{Q}}(x) s_1(x) dy_1 dy_2  dx
             \notag \\
 &= \epsilon^{-2n}  \dint_{\mathbb{R}^{n}}  \dint_{\mathbb{R}^{n}}\dint_{\mathbb{R}^{n}}  (b_{1}(x)-b_1(y_1))  K(x-y_1,x-y_2)  \left(\sum_{k=0}^{\infty} a_{k} \mathrm{e}^{\mathi\frac{\epsilon}{l}(v_{1k},v_{2k})\cdot(x-y_1,x-y_2)} \right)    \notag\\
 &\hspace{8em}   \times  \chi_{_{Q_{1}^{*}}}(y_1) \chi_{_{Q_{2}^{*}}}(y_2) \chi_{_{Q}}(x) s_1(x) dy_1 dy_2  dx
           \notag  \\
 &= \epsilon^{-2n} \sum_{k=0}^{\infty} a_{k}    \dint_{\mathbb{R}^{n}}  \dint_{\mathbb{R}^{n}}\dint_{\mathbb{R}^{n}}  (b_{1}(x)-b_1(y_1))  K(x-y_1,x-y_2) \mathrm{e}^{-\mathi\frac{\epsilon}{l}v_{1k}y_1} \chi_{_{Q_{1}^{*}}}(y_1)    \notag\\
 &\hspace{8em}   \times  \mathrm{e}^{-\mathi\frac{\epsilon}{l}v_{2k}y_2}\chi_{_{Q_{2}^{*}}}(y_2) \mathrm{e}^{\mathi\frac{\epsilon}{l}(v_{1k}+v_{2k})x}   \chi_{_{Q}}(x) s_1(x) dy_1 dy_2  dx.                   \notag
 \end{align}  

Set $f_{1k}(y_1)=\mathrm{e}^{-\mathi\epsilon v_{1k}y_1/l} \chi_{_{Q_{1}^{*}}}(y_1)$, $f_{2k}(y_2)=\mathrm{e}^{-\mathi\epsilon v_{2k}y_2/l}\chi_{_{Q_{2}^{*}}}(y_2)$, $h_{k}(x)= \mathrm{e}^{\mathi\epsilon (v_{1k}+v_{2k})x/l}   \chi_{_{Q}}(x)$, then
\begin{align} \label{equ:thm1-2}
\dint_Q &|b_{1}(x)-(b_1)_{Q_{1}^{*}}| dx =  \epsilon^{-2n}\sum_{k=0}^{\infty} a_{k}   \dint_{\mathbb{R}^{n}}  \dint_{\mathbb{R}^{n}}\dint_{\mathbb{R}^{n}} (b_{1}(x)-b_1(y_1))  K(x-y_1,x-y_2)  \notag\\
 &\hspace{8em}   \times  f_{1k}(y_1)  f_{2k}(y_2) h_{k}(x) s_1(x) dy_1 dy_2  dx      \notag\\
 &=   \epsilon^{-2n}\sum_{k=0}^{\infty} a_{k}   \dint_{\mathbb{R}^{n}} \left( \dint_{\mathbb{R}^{n}}\dint_{\mathbb{R}^{n}} (b_{1}(x)-b_1(y_1))  K(x-y_1,x-y_2) f_{1k}(y_1)  f_{2k}(y_2) dy_1 dy_2 \right)    \notag\\
 &\hspace{8em}   \times   h_{k}(x) s_1(x) dx                 \notag\\
 &= \epsilon^{-2n} \sum_{k=0}^{\infty} a_{k}   \dint_{\mathbb{R}^{n}} T_{b_{1}}  ( f_{1k},  f_{2k})(x)   h_{k}(x) s_1(x) dx \\
 &\le C\epsilon^{-2n} \sum_{k=0}^{\infty} |a_{k}|   \dint_{\mathbb{R}^{n}} |T_{b_{1}}  ( f_{1k},  f_{2k})(x)|   |h_{k}(x)| |s_1(x)| dx  \notag \\
 &  \le C \epsilon^{-2n} \sum_{k=0}^{\infty} |a_{k}|   \dint_{\mathbb{R}^{n}} |T_{b_{1}}  ( f_{1k},  f_{2k})(x)|   |h_{k}(x)|  dx             \notag \\
 &\le C\epsilon^{-2n} \sum_{k=0}^{\infty} |a_{k}|   \dint_{Q} |T_{b_{1}}  ( f_{1k},  f_{2k})(x)|    dx.      \notag
 \end{align}

Note that $f_{jk}(y_j)\le \chi_{_{Q_{j}^{*}}}(y_j)$ for every $y_j \in Q_{j}^{*} $, which implies that $f_{jk}\in L^{p_j(\cdot)}(\mathbb{R}^{n}) ~(j=1,2)$ for every $k\in \mathbb{N}$.
Then, from the generalized H\"{o}lder's inequality (\ref{equ:holder-1}) and the hypothesis, we obtain that
\begin{align*}  
 \dint_{Q} |T_{b_{1}}  ( f_{1k},  f_{2k})(x)|    dx  &\le C  \|T_{b_{1}}  ( f_{1k},  f_{2k}) \|_{{q(\cdot)}} \|\chi_{_{Q}}\|_{{q'(\cdot)}}     \\              
& \le C  \|f_{1k} \|_{{p_{1}(\cdot)}} \|f_{2k} \|_{{p_{2}(\cdot)}} \|\chi_{_{Q}}\|_{{q'(\cdot)}} .  
\end{align*}
Hence, since $Q,Q_{j}^{*} \subset Q_{j0}=Q(x_0,(|z_{j}^*|+1)l)~(j=1,2)$, we have that   
\begin{align*}
\dint_Q |b_{1}(x)-(b_1)_{Q_{1}^{*}}| dx  &\le C \epsilon^{-2n} \sum_{k=0}^{\infty} |a_{k}|~  \|f_{1k} \|_{{p_{1}(\cdot)}} \|f_{2k} \|_{{p_{2}(\cdot)}} \|\chi_{_{Q}}\|_{{q'(\cdot)}}     \\    
& \le C \epsilon^{-2n}  \sum_{k=0}^{\infty} |a_{k}|~  \|\chi_{_{Q_{1}^*}} \|_{{p_{1}(\cdot)}} \|\chi_{_{Q_{2}^*}}  \|_{{p_{2}(\cdot)}} \|\chi_{_{Q}}\|_{{q'(\cdot)}}   \\  
& \le C  \epsilon^{-2n} \sum_{k=0}^{\infty} |a_{k}|~  \|\chi_{_{Q_{10}}} \|_{{p_{1}(\cdot)}} \|\chi_{_{Q_{20}}}  \|_{{p_{2}(\cdot)}} \|\chi_{_{Q}}\|_{{q'(\cdot)}}. 
\end{align*}

Since $1/q(\cdot) = 1/p_{_{1}}(\cdot) + 1/p_{_{2}}(\cdot)  -\delta/n$, and $\frac{2n}{\delta+n} <(p_{j})_{-}<(p_{j})_{+}< \frac{2n}{\delta}~(j=1,2)$,
 then $1/q'(\cdot) = 1+  \frac{\delta}{n}- \frac{1}{p_{_{1}}(\cdot)}  - \frac{1}{p_{_{2}}(\cdot)} =\big(\frac{\delta+n}{2n}- \frac{1}{p_{1}(\cdot)} \big)+ \big(\frac{\delta+n}{2n}- \frac{1}{p_{2}(\cdot)} \big)$.  Thence, by applying the  generalized H\"{o}lder's inequality (\ref{equ:holder-2}), we have
\begin{align*}
 \|\chi_{_{Q}}\|_{{q'(\cdot)}}  &\le     \|\chi_{_{Q}}\|_{\big( \frac{\delta+n}{2n}- \frac{1}{p_{1}(\cdot)} \big) ^{-1}}   \|\chi_{_{Q}}\|_{\big( \frac{\delta+n}{2n}- \frac{1}{p_{2}(\cdot)} \big) ^{-1}}    \\
& \le   \|\chi_{_{Q_{10}}}\|_{\big( \frac{\delta+n}{2n}- \frac{1}{p_{1}(\cdot)} \big) ^{-1}}   \|\chi_{_{Q_{20}}}\|_{\big( \frac{\delta+n}{2n}- \frac{1}{p_{2}(\cdot)} \big) ^{-1}}.
\end{align*}

%

For any $j=1,2$,
 by denoting  $h'_{j}= \big( \frac{\delta+n}{2n}- \frac{1}{p_{j}(\cdot)} \big)^{-1}$, we get that   $\frac{1}{h'_{j}}=  \frac{\delta+n}{2n}- \frac{1}{p_{j}(\cdot)} $, that is, $\frac{1}{p_{j}(\cdot)}  = \frac{1}{h_{j}} - \big(1-\frac{\delta+n}{2n} \big) = \frac{1}{h_{j}} - \big(n-\frac{\delta+n}{2} \big)/n  $.  Thus, since $p_1(\cdot), p_2(\cdot)\in \mathscr{C}^{\log}(\R^{n}) \cap\mathscr{P}(\R^{n})$,
 using \Cref{lem.weight-apq-1}, we obtain that  $1\in \mathcal{A}_{\big[\big( \frac{\delta+n}{2n}- \frac{1}{p_{j}(\cdot)}  \big)^{-1} \big]',p_{j}(\cdot) }^{n-\frac{\delta+n}{2}}(\mathbb{R}^{n}) ~(j=1,2)$.  Then, from \eqref{equ:weight-apq}, doubling condition  implied in conditions  \labelcref{equ:character-norm-est,equ.double-condition-2},  we have
\begin{align*}
\|\chi_{_{Q_{10}}} \|_{{p_{1}(\cdot)}} \|\chi_{_{Q_{20}}}  \|_{{p_{2}(\cdot)}} \|\chi_{_{Q}}\|_{{q'(\cdot)}}  &\le \|\chi_{_{Q_{10}}} \|_{{p_{1}(\cdot)}}     \|\chi_{_{Q_{10}}}\|_{\big( \frac{\delta+n}{2n}- \frac{1}{p_{1}(\cdot)} \big) ^{-1}} \\
 & \hspace{2em} \times \|\chi_{_{Q_{20}}}  \|_{{p_{2}(\cdot)}} \|\chi_{_{Q_{20}}}\|_{\big( \frac{\delta+n}{2n}- \frac{1}{p_{2}(\cdot)} \big) ^{-1}}   \\
& \le  C|Q_{10}|^{\frac{\delta+n}{2n}}  |Q_{20}|^{\frac{\delta+n}{2n}}     \\
& \le  C|Q|^{\frac{\delta+n}{n}}.
\end{align*}
Thus
\begin{align*}
\dint_Q |b_{1}(x)-(b_1)_{Q_{1}^{*}}| dx  & \le C  \epsilon^{-2n} \sum_{k=0}^{\infty} |a_{k}|~  \|\chi_{_{Q_{10}}} \|_{{p_{1}(\cdot)}} \|\chi_{_{Q_{20}}}  \|_{{p_{2}(\cdot)}} \|\chi_{_{Q}}\|_{{q'(\cdot)}}       \\
&\le C \epsilon^{-2n} |Q|^{\frac{\delta+n}{n}} \sum_{k=0}^{\infty} |a_{k}|     \\
& \le C  |Q|^{\frac{\delta+n}{n}}.
\end{align*}
Since ${\vec{z}}^{^{*}} =(z_{1}^{*},z_{2}^{*})$ is fixed, by taking supremum over every $Q\subset \mathbb{R}^n$, we obtain that $b_{1}\in \mathbb{L}(\delta) $. Similar argument as above, we can get that $b_{2}\in \mathbb{L}(\delta) $. Therefore, $\vec{b}=(b_{1},b_{2})\in \mathbb{L}(\delta) \times \mathbb{L}(\delta) $.

$\xLongrightarrow{\ \ \ \ }$:\ Now, we   prove that the necessary condition.
Assume that $\vec{b}=(b_{1},b_{2})\in \mathbb{L}(\delta) \times \mathbb{L}(\delta) $ and $\vec{f}=(f_{1},f_{2})\in L^{p_1(\cdot)}(\mathbb{R}^{n}) \times L^{p_2(\cdot)}(\mathbb{R}^{n}) $. 

Then, by the definition of Lipschitz's space ( see (1) of   \cref{def.lip-space}),  condition \eqref{equ:CZK-1}, the monotonically increasing of function $t^{\delta}$ for any $t>0$, and taking into account that $\|f\|_{\Lambda_{\delta}} \approx \|f\|_{\mathbb{L}(\delta)}$, 
we can obtain
\begin{align*}
|T_{b_{1}} (\vec{f}) (x)|  &\le \int_{(\mathbb{R}^{n})^{2}} |b_{1}(x)-b_1(y_1)|  |K(x,\vec{y})|  |f_{1}(y_1)|  |f_{2}(y_2)| d\vec{y}  \\
& \le C \|b_{1}\|_{\mathbb{L}(\delta)} \int_{(\mathbb{R}^{n})^{2}}   \frac{|x-y_1|^{\delta}}{(|x-y_1|+|x-y_2|)^{2n}}  |f_{1}(y_1)|  |f_{2}(y_2)|  d\vec{y} \\
& \le C \|b_{1}\|_{\mathbb{L}(\delta)} \int_{(\mathbb{R}^{n})^{2}}   \frac{|f_{1}(y_1)|  |f_{2}(y_2)| }{(|x-y_1|+|x-y_2|)^{2n-\delta}}   d\vec{y}  \\
& = C \|b_{1}\|_{\mathbb{L}(\delta)} \mathcal{I}_{\delta}(|f_{1}|,|f_{2}|)(x) .
\end{align*}

Since $\frac{1}{p_{_{1}}(x)} +\frac{1}{p_{_{2}}(x)} -\frac{1}{q(x)} = \frac{\delta}{n}$, $0< \delta<1$, $p_1(\cdot),p_2(\cdot)\in \mathscr{C}^{\log}(\R^{n}) \cap\mathscr{P}(\R^{n})$ and $(p_1)_{+},(p_2)_{+}<\frac{2n}{\delta}$.   Thus, using \cref{lem. multilinear-fractional}, we have
\begin{align*}
\|T_{b_{1}} (\vec{f})\|_{q(\cdot)}  &\le C \|b_{1}\|_{\mathbb{L}(\delta)} \|\mathcal{I}_{\delta}(|f_{1}|,|f_{2}|)\|_{q(\cdot)} \\
& \le  C \|b_{1}\|_{\mathbb{L}(\delta)} \|f_{1}\|_{p_{_{1}}(\cdot)} \| f_{2}\|_{p_{_{2}}(\cdot)}.  
\end{align*}


Similar argument as above, we can get that $\|T_{b_{2}} (\vec{f})\|_{q(\cdot)} \le  C \|b_{2}\|_{\mathbb{L}(\delta)} \|f_{1}\|_{p_{_{1}}(\cdot)} \| f_{2}\|_{p_{_{2}}(\cdot)} $

Combining these estimates, the proof is completed.
\end{proof}


The following theorem characterizes the variable spaces $\mathbb{L}(\delta(\cdot)) $  in terms of the boundedness of $T_{_{\Sigma \vec{b}}}$.

\begin{theorem}\label{thm.2-0}
Suppose that $0< \delta(\cdot)<1$,  $0<\gamma<\eta<1/m$,  $ \epsilon_{0}<\epsilon \le 1$, $r(\cdot),p_{_{1}}(\cdot),p_{_{2}}(\cdot),\dots,p_{_{m}}(\cdot)\in \mathscr{C}^{\log}(\R^{n}) \cap\mathscr{P}(\R^{n})$ satisfy $\frac{1}{p(x)}= \frac{1}{p_{_{1}}(x)} + \frac{1}{p_{_{2}}(x)}+\cdots+\frac{1}{p_{_{m}}(x)}$, $r(x)\ge r_\infty$ for almost every $x\in \mathbb{R}^{n}$ and $1<\beta\le r_{-}$ such that $\frac{1}{p(x)}-\frac{1}{q(x)}=\frac{\delta(x)}{n}=\frac{1}{\beta}-\frac{1}{r(x)}$ with $\sup_{x\in \mathbb{R}^{n}} {p_{_j}}(x) \delta(x)<n~(j=1,2,\dots,m)$.
Then $\vec{b}=(b_{1},b_{2},\dots,b_{m})\in \mathbb{L}(\delta(\cdot)) \times\mathbb{L}(\delta(\cdot)) \times\cdots\times \mathbb{L}(\delta(\cdot)) $ and $\|T_{_{b_{j}}} \vec{f} \|_{q(\cdot)}<\infty $ for $ \vec{f}=(f_{1},f_{2},\dots,f_{m})\in  L^{p_1(\cdot)}(\mathbb{R}^{n}) \times L^{p_2(\cdot)}(\mathbb{R}^{n}) \times\cdots\times L^{p_m(\cdot)}(\mathbb{R}^{n})$
 if and only if
$T_{_{b_{j}}}: L^{p_1(\cdot)}(\mathbb{R}^{n}) \times L^{p_2(\cdot)}(\mathbb{R}^{n}) \times\cdots\times L^{p_m(\cdot)}(\mathbb{R}^{n}) \to L^{q(\cdot)}(\mathbb{R}^{n})~(j=1,2,\dots,m)$.

\end{theorem}

%
%

\begin{proof}
Without loss of generality, we only consider the case that $m=2$. Actually, similar procedure work for all $m\in \mathbb{N}$.

$\xLongleftarrow{\ \ \ \ }$:\  We first prove that the condition is  sufficient.
Assume that $T_{_{b_{1}}}$ maps $L^{p_1(\cdot)}(\mathbb{R}^{n}) \times L^{p_2(\cdot)}(\mathbb{R}^{n})$ into $ L^{q(\cdot)}(\mathbb{R}^{n})$. By proceeding as in \eqref{equ:thm1-2} with $Q=Q(x_{0},\ell)$ and $Q_{j}^{*} = Q(y_{j0},\ell) \subset \mathbb{R}^{n} ~(j=1,2)$, we have
\begin{align*}
 \dint_{Q} |b_{1}(x)-(b_1)_{Q_{1}^{*}}| dx   &\le C\epsilon^{-2n} \sum_{k=0}^{\infty} |a_{k}|   \dint_{Q} |T_{b_{1}}  ( f_{1k},  f_{2k})(x)|    dx,
\end{align*}
where $f_{jk}\in L^{p_j(\cdot)}(\mathbb{R}^{n}) $ and $\|f_{jk} \|_{{p_{j}(\cdot)}} \le \|\chi_{_{Q_{j}^{*}}}\|_{{p_{j}(\cdot)}} ~(j=1,2)$ for every $k\in \mathbb{N}$.
Then, from the generalized H\"{o}lder's inequality (\ref{equ:holder-1}) and the hypothesis, we obtain that
\begin{align*}
 \dint_{Q} |T_{b_{1}}  ( f_{1k},  f_{2k})(x)|    dx  &\le C  \|T_{b_{1}}  ( f_{1k},  f_{2k}) \|_{{q(\cdot)}} \|\chi_{_{Q}}\|_{{q'(\cdot)}}     \\
& \le C  \|f_{1k} \|_{{p_{1}(\cdot)}} \|f_{2k} \|_{{p_{2}(\cdot)}} \|\chi_{_{Q}}\|_{{q'(\cdot)}},
\end{align*}
and then
\begin{align*}
 \dint_{Q} |b_{1}(x)-(b_1)_{Q_{1}^{*}}| dx   &\le C\epsilon^{-2n} \sum_{k=0}^{\infty} |a_{k}|  \|\chi_{_{Q_{1}^{*}}} \|_{{p_{1}(\cdot)}}  \|  \chi_{_{Q_{2}^{*}}} \|_{{p_{2}(\cdot)}} \|\chi_{_{Q}}\|_{{q'(\cdot)}}.  
\end{align*}

Set $Q,Q_{j}^{*} \subset Q_{j0}=Q(x_0,(|z_{j}^*|+1)\ell)~(j=1,2)$. Since  
$\frac{1}{p(\cdot)}-\frac{1}{q(\cdot)}=\frac{\delta(\cdot)}{n}=\frac{1}{\beta}-\frac{1}{r(\cdot)}$ , then $1/q'(\cdot) = 1/r'(\cdot) + ( 1/\beta- 1/p(\cdot) ) = \frac{1}{r'(\cdot) }+ \big( \frac{1}{2\beta} - \frac{1}{p_{1}(\cdot)} \big)  + \big( \frac{1}{2\beta} - \frac{1}{p_{2}(\cdot)} \big)$.  Thence, by applying the  generalized H\"{o}lder's inequality (\ref{equ:holder-2}), we get that
\begin{align*}
 \|\chi_{_{Q}}\|_{{q'(\cdot)}}  &\le    \|\chi_{_{Q}}\|_{{r' (\cdot)}} \|\chi_{_{Q}}\|_{\big( \frac{1}{2\beta} - \frac{1}{p_{1}(\cdot)} \big) ^{-1}}   \|\chi_{_{Q}}\|_{\big( \frac{1}{2\beta} - \frac{1}{p_{2}(\cdot)} \big) ^{-1}} \\
& \le   \|\chi_{_{Q_{1}^{*}}}\|_{\big( \frac{1}{2\beta} - \frac{1}{p_{1}(\cdot)} \big) ^{-1}}   \|\chi_{_{Q_{2}^{*}}}\|_{\big( \frac{1}{2\beta} - \frac{1}{p_{2}(\cdot)} \big) ^{-1}}   \|\chi_{_{Q}}\|_{{r' (\cdot)}}    \\
& \le   \|\chi_{_{Q_{10}}}\|_{\big( \frac{1}{2\beta} - \frac{1}{p_{1}(\cdot)} \big) ^{-1}}   \|\chi_{_{Q_{20}}}\|_{\big( \frac{1}{2\beta} - \frac{1}{p_{2}(\cdot)} \big) ^{-1}}   \|\chi_{_{Q}}\|_{{r' (\cdot)}}.
\end{align*}

For any $j=1,2$, let $h'_{j}= \big( \frac{1}{2\beta} - \frac{1}{p_{j}(\cdot)} \big)^{-1}$, we get that $\frac{1}{h'_{j}}=  \frac{1}{2\beta} - \frac{1}{p_{j}(\cdot)} $,  that is, $\frac{1}{p_{j}(\cdot)}  = \frac{1}{h_{j}} - \big(1-\frac{1}{2\beta}  \big) = \frac{1}{h_{j}} - \big(n-\frac{n}{2\beta} \big)/n  $. Thus, since $r(\cdot), p_{1}(\cdot), p_{2}(\cdot)\in   \mathscr{C}^{\log}(\R^{n}) \cap\mathscr{P}(\R^{n})$,
 using \Cref{lem.weight-apq-1}, we obtain that $1\in \mathcal{A}_{\big[\big( \frac{1}{2\beta} - \frac{1}{p_{j}(\cdot)} \big)^{-1} \big]',p_{j}(\cdot) }^{n-\frac{n}{2\beta}}(\mathbb{R}^{n}) ~(j=1,2)$.  Then, from \eqref{equ:weight-apq},  \cref{lem:T8-condition} and doubling condition  implied in conditions \labelcref{equ:character-norm-est,equ.double-condition-2},  we have
 \begin{align*}
\|\chi_{_{Q_{1}^{*}}} \|_{{p_{1}(\cdot)}}  \|  \chi_{_{Q_{2}^{*}}} \|_{{p_{2}(\cdot)}} \|\chi_{_{Q}}\|_{{q'(\cdot)}}  &\le \|\chi_{_{Q_{10}}} \|_{{p_{1}(\cdot)}}  \|  \chi_{_{Q_{20}}} \|_{{p_{2}(\cdot)}} \|\chi_{_{Q}}\|_{{q'(\cdot)}}          \\
 &\le \|\chi_{_{Q_{10}}} \|_{{p_{1}(\cdot)}}     \|\chi_{_{Q_{10}}}\|_{\big( \frac{1}{2\beta} - \frac{1}{p_{1}(\cdot)} \big) ^{-1}} \\
 & \hspace{2em} \times \|\chi_{_{Q_{20}}}  \|_{{p_{2}(\cdot)}} \|\chi_{_{Q_{20}}}\|_{\big( \frac{1}{2\beta} - \frac{1}{p_{2}(\cdot)} \big) ^{-1}}   \|\chi_{_{Q}}\|_{{r' (\cdot)}}  \\
& \le  C|Q_{10}|^{\frac{1}{2\beta} }  |Q_{20}|^{\frac{1}{2\beta}  } ~ \|\chi_{_{Q}}\|_{{r'(\cdot)}}    \\
& \le  C|Q|^{1/\beta}~ \|\chi_{_{Q}}\|_{{r'(\cdot)}}.
\end{align*}
 Hence,
\begin{align*}
\dint_Q |b_{1}(x)-(b_1)_{Q_{1}^{*}}| dx  &\le C \epsilon^{-2n} \sum_{k=0}^{\infty} |a_{k}|~  |Q|^{1/\beta}~ \|\chi_{_{Q}}\|_{{r'(\cdot)}}    \\
& \le C   |Q|^{1/\beta}~ \|\chi_{_{Q}}\|_{{r'(\cdot)}}.
\end{align*}
Since ${\vec{z}}^{^{*}} =(z_{1}^{*},z_{2}^{*})$ is fixed, by taking supremum over every $Q\subset \mathbb{R}^n$, we obtain that $b_{1}\in \mathbb{L}(\delta(\cdot)) $. Similar argument as above, we can get that $b_{2}\in \mathbb{L}(\delta(\cdot)) $. 

$\xLongrightarrow{\ \ \ \ }$:\ Now, we   prove that the necessary condition.
Assume that $\vec{b}=(b_{1},b_{2})\in \mathbb{L}(\delta(\cdot)) \times \mathbb{L}(\delta(\cdot)) $ and $\vec{f}=(f_{1},f_{2})\in L^{p_1(\cdot)}(\mathbb{R}^{n}) \times L^{p_2(\cdot)}(\mathbb{R}^{n}) $ such that $\|T_{_{b_{j}}}  (\vec{f}) \|_{q(\cdot)}<\infty ~(j=1,2)$.
 From  \Cref{lem.sharp-estimate-var-1}, for any given  $\gamma\in (0,1)$, it follows that
\begin{align*}
 \|T_{_{b_{j}}}  (\vec{f}) \|_{q(\cdot)}     &\le C \|(T_{_{b_{j}}}  (\vec{f}))_{\delta(\cdot),\gamma}^{\sharp} \|_{p(\cdot)}.
\end{align*}

Since $ \frac{1}{k_{0}}= \epsilon_{0}<\epsilon \le 1$,  $0<\gamma<\eta<1/m$,    by using \Cref{lem.sharp-estimate-var}, we have that
\begin{align*}
 \|T_{_{b_{j}}} (\vec{f}) \|_{q(\cdot)}     &\le C \|(T_{_{b_{j}}} (\vec{f}))_{\delta(\cdot),\gamma}^{\sharp} \|_{p(\cdot)} \\
  &\le C \|b_{j}\|_{\mathbb{L}(\delta(\cdot))} \Big(\|M_{\eta}(T\vec{f})\|_{p(\cdot)}+\|\mathcal{M}_{L(\log L)}(\vec{f})\|_{p(\cdot)} \Big).
\end{align*}
Recall the pointwise equivalence $M_{L(\log L)}(g)(x)\approx M^{2}(g)(x)$ for any locally integrable function $g$ \cite[see (21) in ][]{perez1995endpoint} and
\begin{align*}
\mathcal{M}_{L(\log L)}(\vec{f})(x)  &\le \prod_{i=1}^{2} \Big( \sup\limits_{Q\ni x}    \|f_{i}\|_{L(\log L) ,Q}\Big)  =\prod_{i=1}^{2} M_{L(\log L)}(f_{i})(x),
\end{align*}
then, by the generalized H\"{o}lder's inequality (\cref{lem.holder}) and the hypothesis on $p(\cdot)$, there has
\begin{align*}
\|\mathcal{M}_{L(\log L)}(\vec{f})\|_{p(\cdot)}  &\le C \Big\|\prod_{i=1}^{2} M_{L(\log L)}(f_{i}) \Big\|_{p(\cdot)}  \le C \Big\|\prod_{i=1}^{2} M^{2}(f_{i}) \Big\|_{p(\cdot)} \le C \prod_{i=1}^{2} \|f_{i} \|_{p_{i}(\cdot)},
\end{align*}
where in the last inequality, we make use of the $L^{p_{i}(\cdot)}(\mathbb{R}^{n})$ boundedness of $M$ twice.

In addition,  apply \cref{lem.variable-proposition-B,lem.variable-homogeneous,lem.holder,lem.CZO-bound}, it is easy to see that $ p_{_{1}}(\cdot),~p_{_{2}}(\cdot) \in \mathscr{B}(\R^{n})$, and that $ p(\cdot)\in    \mathscr{P}(\R^{n})$  satisfies $(p(\cdot)/p_{0})' \in \mathscr{B}(\R^{n})$ with some $p_{0}\in (0,p_{-})$, then we have
\begin{align*}
 \|M_{\eta}(T\vec{f})\|_{p(\cdot)}     &\le C \Big\|  |T(\vec{f})|^{\eta} \Big\|_{p(\cdot)/\eta}^{1/\eta}
   \le C \Big\|   T(\vec{f})  \Big\|_{p(\cdot)}    \le C \prod_{i=1}^{2} \|f_{i} \|_{p_{i}(\cdot)}.
\end{align*}

Thus, we obtain that
\begin{align*}
 \|T_{_{b_{j}}} (\vec{f}) \|_{q(\cdot)}     &\le   C  \|b_{j}\|_{\mathbb{L}(\delta(\cdot))} \|f_{1}\|_{p_{1}(\cdot)}\|f_{2}\|_{p_{2}(\cdot)}.
\end{align*}

Combining these estimates, the proof is completed.
\end{proof}

\bigskip


\begin{thebibliography}{99}

\bibitem{difazio1993interior}
G.~Di~Fazio and M.~A. Ragusa, ``Interior estimates in {Morrey} spaces for
  strong solutions to nondivergence form equations with discontinuous
  coefficients,'' {\em Journal of Functional Analysis}, vol.~112, no.~2,
  pp.~241--256, 1993.

\bibitem{bramanti1995commutators}
M.~Bramanti and M.~C. Cerutti, ``Commutators of singular integrals and
  fractional integrals on homogeneous spaces,'' {\em Contemporary Mathematics},
  vol.~189, pp.~81--81, 1995.

\bibitem{rios2003lp}
C.~Rios, ``The {$L^{p}$} {Dirichlet} problem and nondivergence harmonic
  measure,'' {\em Transactions of the American Mathematical Society}, vol.~355,
  no.~2, pp.~665--687, 2003.

\bibitem{bernardis2006weighted}
A.~L. Bernardis, S.~Hartzstein, and G.~G. Pradolini, ``Weighted inequalities
  for commutators of fractional integrals on spaces of homogeneous type,'' {\em
  Journal of mathematical analysis and applications}, vol.~322, no.~2,
  pp.~825--846, 2006.

\bibitem{cruz2003endpoint}
D.~Cruz-Uribe and A.~Fiorenza, ``Endpoint estimates and weighted norm
  inequalities for commutators of fractional integrals,'' {\em Publicacions
  matem\`{a}tiques}, vol.~47, no.~1, pp.~103--131, 2003.

\bibitem{perez1995endpoint}
C.~P{\'e}rez, ``Endpoint estimates for commutators of singular integral
  operators,'' {\em Journal of functional analysis}, vol.~128, no.~1,
  pp.~163--185, 1995.

\bibitem{perez2014end}
C.~P{\'e}rez, G.~G. Pradolini, R.~H. Torres, and R.~Trujillo-Gonz{\'a}lez,
  ``End-point estimates for iterated commutators of multilinear singular
  integrals,'' {\em Bulletin of the London Mathematical Society}, vol.~46,
  no.~1, pp.~26--42, 2014.

\bibitem{coifman1976factorization}
R.~R. Coifman, R.~Rochberg, and G.~Weiss, ``Factorization theorems for {Hardy}
  spaces in several variables,'' {\em Annals of Mathematics}, vol.~103, no.~3,
  pp.~611--635, 1976.

\bibitem{janson1978mean}
S.~Janson, ``Mean oscillation and commutators of singular integral operators,''
  {\em Arkiv f{\"o}r Matematik}, vol.~16, no.~1, pp.~263--270, 1978.

\bibitem{chanillo1982note}
S.~Chanillo, ``A note on commutators,'' {\em Indiana University Mathematics
  Journal}, vol.~31, no.~1, pp.~7--16, 1982.

\bibitem{paluszynski1995characterization}
M.~Paluszy{\'n}ski, ``Characterization of the {Besov} spaces via the commutator
  operator of {Coifman}, {Rochberg} and {Weiss},'' {\em Indiana University
  Mathematics Journal}, pp.~1--17, 1995.

\bibitem{coifman1975commutators}
R.~R. Coifman and Y.~Meyer, ``On commutators of singular integrals and bilinear
  singular integrals,'' {\em Transactions of the American Mathematical
  Society}, vol.~212, pp.~315--331, 1975.

\bibitem{coifman1978commutateurs}
R.~R. Coifman and Y.~Meyer, ``Commutateurs d'int{\'e}grales singuli{\`e}res et
  op{\'e}rateurs multilin{\'e}aires,'' in {\em Annales de l'institut Fourier},
  vol.~28, pp.~177--202, 1978.

\bibitem{grafakos2002multilinear}
L.~Grafakos and R.~H. Torres, ``Multilinear {Calder{\'o}n--Zygmund} theory,''
  {\em Advances in Mathematics}, vol.~165, no.~1, pp.~124--164, 2002.

\bibitem{lerner2009new}
A.~K. Lerner, S.~Ombrosi, C.~P{\'e}rez, R.~H. Torres, and
  R.~Trujillo-Gonz{\'a}lez, ``New maximal functions and multiple weights for
  the multilinear {Calder{\'o}n--Zygmund} theory,'' {\em Advances in
  Mathematics}, vol.~220, no.~4, pp.~1222--1264, 2009.

\bibitem{grafakos2014multilinear}
L.~Grafakos, L.~Liu, D.~Maldonado, and D.~Yang, {\em Multilinear analysis on
  metric spaces}.
\newblock Citeseer, 2014.

\bibitem{kovavcik1991spaces}
O.~Kov{\'a}{\v{c}}ik and J.~R{\'a}kosn{\'\i}k, ``On spaces ${L^{p(x)}}$ and
  ${W^{k,p(x)}}$,'' {\em Czechoslovak Mathematical Journal}, vol.~41, no.~4,
  pp.~592--618, 1991.

\bibitem{diening2003calderon}
L.~Diening and M.~R\r{u}\u{z}i\u{c}ka, ``Calder{\'o}n-{Z}ygmund operators on
  generalized {Lebesgue} spaces {$L^{p(\cdot)}$} and problems related to fluid
  dynamics,'' {\em Journal f\"{u}r die reine und angewandte Mathematik},
  vol.~563, pp.~197--220, 2003.

\bibitem{cruz2006theboundedness}
D.~Cruz-Uribe, A.~Fiorenza, J.~M. Martell, and C.~P{\'e}rez, ``The boundedness
  of classical operators on variable ${L^{p}}$ spaces,'' {\em Annales
  Academi{\ae} Scientiarum Fennic{\ae} Mathematica}, vol.~31, no.~1,
  pp.~239--264, 2006.

\bibitem{diening2011lebesgue}
L.~Diening, P.~Harjulehto, P.~H{\"a}st{\"o}, and M.~Ruzicka, {\em Lebesgue and
  {Sobolev} spaces with variable exponents}, vol.~2017 of {\em Lecture Notes in
  Mathematics}.
\newblock Berlin,Heidelberg: Springer-Verlag, 2011.

\bibitem{cruz2013variable}
D.~Cruz-Uribe and A.~Fiorenza, {\em Variable {Lebesgue} spaces: foundations and
  harmonic analysis}.
\newblock Springer Science \& Business Media, Heidelberg, 2013.

\bibitem{ramseyer2013lipschitz}
M.~J. Ramseyer, O.~Salinas, and B.~Viviani, ``Lipschitz type smoothness of the
  fractional integral on variable exponent spaces,'' {\em Journal of
  Mathematical Analysis and Applications}, vol.~403, no.~1, pp.~95--106, 2013.

\bibitem{cabral2016extrapolation}
A.~Cabral, G.~G. Pradolini, and W.~A. Ramos, ``Extrapolation and weighted norm
  inequalities between {Lebesgue} and {Lipschitz} spaces in the variable
  exponent context,'' {\em Journal of Mathematical Analysis and Applications},
  vol.~436, no.~1, pp.~620--636, 2016.

\bibitem{pradolini2017characterization}
G.~G. Pradolini and W.~A. Ramos, ``Characterization of {Lipschitz} functions
  via the commutators of singular and fractional integral operators in variable
  {Lebesgue} spaces,'' {\em Potential Analysis}, vol.~46, no.~3, pp.~499--525,
  2017.

\bibitem{huang2010multilinear}
A.-w. Huang and J.-s. Xu, ``Multilinear singular integrals and commutators in
  variable exponent {Lebesgue} spaces,'' {\em Applied Mathematics-A Journal of
  Chinese Universities}, vol.~25, no.~1, pp.~69--77, 2010.

\bibitem{lu2014multilinear}
G.~Lu and P.~Zhang, ``Multilinear {Calder\'{o}n-Zygmund} operators with kernels
  of {Dini's} type and applications,'' {\em Nonlinear Analysis: Theory,Methods
  and Applications}, vol.~107, pp.~92--117, 2014.

\bibitem{xu2006generalized}
J.~S. Xu, ``Generalized commutators of multilinear singular integrals,'' {\em
  Proceedings of A. Razmadze Mathematical Institute}, vol.~142, no.~1,
  pp.~109--119, 2006.

\bibitem{cruz2003maximal}
D.~Cruz-Uribe, A.~Fiorenza, and C.~J. Neugebauer, ``The maximal function on
  variable ${L^{p}}$ spaces,'' {\em Annales Academi{\ae} Scientiarum
  Fennic{\ae} Mathematica}, vol.~28, no.~1, pp.~223--238, 2003.

\bibitem{cruz2014variable}
D.~Cruz-Uribe and L.-A.~D. Wang, ``Variable {H}ardy spaces,'' {\em Indiana
  university mathematics journal}, vol.~63, no.~2, pp.~447--493, 2014.

\bibitem{nekvinda2004hardy}
A.~Nekvinda, ``Hardy-{L}ittlewood maximal operator on
  ${L}^{p(x)}(\mathbb{R}^{n})$,'' {\em Mathematical Inequalities and
  Applications}, vol.~7, pp.~255--266, 2004.

\bibitem{diening2005maximalf}
L.~Diening, ``Maximal function on {M}usielak-{O}rlicz spaces and generalized
  {Lebesgue} spaces,'' {\em Bulletin des Sciences Math{\'e}matiques}, vol.~129,
  no.~8, pp.~657--700, 2005.

\bibitem{perez2002sharp}
C.~P{\'e}rez and R.~Trujillo-Gonz\'{a}lez, ``Sharp weighted estimates for
  multilinear commutators,'' {\em Journal of the London mathematical society},
  vol.~65, no.~3, pp.~672--692, 2002.

\bibitem{cruz2012weighted}
D.~Cruz-Uribe, A.~Fiorenza, and C.~J. Neugebauer, ``Weighted norm inequalities
  for the maximal operator on variable {Lebesgue} spaces,'' {\em Journal of
  Mathematical Analysis and Applications}, vol.~394, no.~2, pp.~744--760, 2012.

\bibitem{harjulehto2004variable}
P.~Harjulehto, P.~H{\"a}st{\"o}, and M.~Pere, ``Variable exponent {Lebesgue}
  spaces on metric spaces: the {Hardy}-{Littlewood} maximal operator,'' {\em
  Real Analysis Exchange}, vol.~30, no.~1, pp.~87--104, {2004}.

\bibitem{fan2001spaces}
X.~Fan and D.~Zhao, ``On the spaces {$L^{p(x)}(\Omega)$ and $W^{m, p
  (x)}(\Omega)$},'' {\em Journal of Mathematical Analysis and Applications},
  vol.~263, no.~2, pp.~424--446, 2001.

\bibitem{cruz2011maximal}
D.~Cruz-Uribe, L.~Diening, and P.~H{\"a}st{\"o}, ``The maximal operator on
  weighted variable {Lebesgue} spaces,'' {\em Fractional Calculus and Applied
  Analysis}, vol.~14, no.~3, pp.~361--374, 2011.

\bibitem{bernardis2014generalized}
A.~L. Bernardis, E.~D. Dalmasso, and G.~G. Pradolini, ``Generalized maximal
  functions and related operators on weighted {Musielak}-{Orlicz} spaces,''
  {\em Annales Academi{\ae} Scientiarum Fennic{\ae} Mathematica}, vol.~39,
  pp.~23--50, 2014.

\bibitem{kenig1999multilinear}
C.~E. Kenig and E.~M. Stein, ``Multilinear estimates and fractional
  integration,'' {\em Mathematical Research Letters}, vol.~6, no.~1, pp.~1--15,
  1999.

\bibitem{moen2009weighted}
K.~Moen, ``Weighted inequalities for multilinear fractional integral
  operators,'' {\em Collectanea mathematica}, vol.~60, no.~2, pp.~213--238,
  2009.

\bibitem{pradolini2010weighted}
G.~G. Pradolini, ``Weighted inequalities and pointwise estimates for the
  multilinear fractional integral and maximal operators,'' {\em Journal of
  Mathematical Analysis and Applications}, vol.~367, no.~2, pp.~640 -- 656,
  2010.

\bibitem{tan2017some}
J.~Tan, Z.~Liu, and J.~Zhao, ``On some multilinear commutators in variable
  {Lebesgue} spaces,'' {\em Journal of Mathematical Inequalities}, vol.~11,
  no.~3, pp.~715--734, 2017.

\bibitem{pradolini2001class}
G.~G. Pradolini, ``A class of pairs of weights related to the boundedness of
  the fractional integral operator between ${L}^{p}$ and {Lipschitz} spaces,''
  {\em Commentationes Mathematicae Universitatis Carolinae}, vol.~42, no.~1,
  pp.~133--152, 2001.

\bibitem{pradolini2001two}
G.~G. Pradolini, ``Two-weighted norm inequalities for the fractional integral
  operator between ${L}^{p}$ and {Lipschitz} spaces,'' {\em Commentationes
  Mathematicae}, vol.~41, pp.~147--169, 2001.

\bibitem{harboure1997boundedness}
E.~Harboure, O.~Salinas, and B.~Viviani, ``Boundedness of the fractional
  integral on weighted {Lebesgue} and {Lipschitz} spaces,'' {\em Transactions
  of the American Mathematical Society}, vol.~349, no.~1, pp.~235--255, 1997.

\bibitem{john1961functions}
F.~John and L.~Nirenberg, ``On functions of bounded mean oscillation,'' {\em
  Communications on pure and applied Mathematics}, vol.~14, no.~3,
  pp.~415--426, 1961.

\bibitem{muckenhoupt1976weighted}
B.~Muckenhoupt and R.~Wheeden, ``Weighted bounded mean oscillation and the
  {Hilbert} transform,'' {\em Studia Mathematica}, vol.~3, no.~54,
  pp.~221--237, 1976.

\bibitem{li2005jhon}
W.-m. Li, ``Jhon-{Nirenberg} inequality and self-improving properties,'' {\em
  Journal on Mathematical Research and Exposition}, vol.~25, no.~1, pp.~42--46,
  2005.

\bibitem{garcia1985weighted}
J.~Garc{\'\i}a-Cuerva and J.~L.~R. De~Francia, {\em Weighted norm inequalities
  and related topics}, vol.~116 of {\em North-Holland Mathematics Studies}.
\newblock Elsevier Science Publishers B.V., 1985.

\end{thebibliography}
\end{document}